\documentclass[reqno]{ndb}

\usepackage[margin=1.4in]{geometry}
\usepackage{amsmath}
\usepackage{afterpage}
\usepackage[dvips]{graphicx}
\usepackage[normal]{subfigure}
\usepackage{amssymb}
\usepackage{verbatim}
\usepackage[usenames,dvipsnames]{color}
\usepackage{mathrsfs}
\usepackage{colonequals}
\usepackage{cleveref}
\usepackage{paralist}
\usepackage{ctable}

\newtheorem{theorem}{Theorem}[section]
\newtheorem{nonumbtheorem}{Theorem}

\newtheorem{lemma}[theorem]{Lemma}
\newtheorem{proposition}[theorem]{Proposition}

\theoremstyle{definition}
\newtheorem{definition}[theorem]{Definition}

\theoremstyle{remark}
\newtheorem{remark}[theorem]{Remark}

\newcommand{\R}{\mathbb{R}}

\newcommand{\emdash}{\hspace{1pt}---\hspace{1pt}}

\newcommand{\bigoh}{\mathcal{O}}

\newcommand{\be}{\begin{equation}}
\newcommand{\ee}{\end{equation}}
\newcommand{\bea}{\begin{align}}
\newcommand{\eea}{\end{align}}
\newcommand{\bead}{\begin{aligned}}
\newcommand{\eead}{\end{aligned}}
\newcommand{\bsub}{\begin{subequations}}
\newcommand{\esub}{\end{subequations}}

\DeclareMathOperator{\Div}{\mathrm{div}}

\newcommand{\ce}{\colonequals}

\newcommand{\ud}[1]{\mathop{\mathrm{d}#1}}
\newcommand{\pd}[2]{\frac{\partial #1}{\partial #2}}
\newcommand{\pdd}[2]{\frac{\partial^2 #1}{\partial #2^2}}

\newcommand{\eps}{\varepsilon}

\numberwithin{equation}{section}

\begin{document}
\title{Analysis of a one-dimensional prescribed mean curvature equation with singular nonlinearity}
\author{Nicholas D. Brubaker\footnotemark[1]\ \footnotemark[2] \and John A. Pelesko\footnotemark[1]\ \footnotemark[3]}

\maketitle

\renewcommand{\thefootnote}{\fnsymbol{footnote}}

\footnotetext[1]{Department of Mathematical Sciences, University of Delaware, Newark, DE 19716, USA (Email addresses: {\tt brubaker@math.udel.edu}, {\tt pelesko@math.udel.edu})}
\footnotetext[2]{The work of this author was supported by the National Science Foundation through a Graduate Research Fellowship.}
\footnotetext[3]{The work of this author was supported by the National Science Foundation, Award No. 312154.}

\maketitle

\begin{abstract}
In this paper, the classical  solution set $(\lambda,u)$ of the one-dimensional prescribed mean curvature equation
\be\tag{$\star$}\label{eq:abstract}
\begin{aligned}
 & -\left(\frac{u'}{\sqrt{1+(u')^2}} \right)' = \frac{\lambda}{(1-u)^2},  \ \   -L<x<L;  \qquad   u(-L)=u(L)= 0,
\end{aligned}
\ee
for  $\lambda>0$ and $L>0$, is analyzed via a time map. It is shown that the solution set depends on both parameters $\lambda$ and $L$ and undergoes two bifurcations. The first is a standard saddle node bifurcation, which happens for all $L$ at $\lambda = \lambda^*(L)$. The second is a \emph{splitting} bifurcation, namely, there exists a value $L^*$ such that as $L$ transitions from greater than or equal to $L^*$  to less than $L^*$ the upper branch of the bifurcation diagram of problem \eqref{eq:abstract} splits into two parts. In contrast, the solution set of the semilinear version of problem \eqref{eq:abstract} is independent of $L$ and exhibits only a saddle node bifurcation. Therefore, as this analysis suggests, the splitting bifurcation is a byproduct of the mean curvature operator coupled with the singular forcing.
\end{abstract}



\pagestyle{myheadings}
\thispagestyle{plain}
\markboth{NICHOLAS D. BRUBAKER AND JOHN A. PELESKO}{ANALYSIS OF A ONE-DIMENSIONAL PRESCRIBED MEAN CURVATURE EQUATION}

\section{Introduction} 

The  study of nonparametric surfaces of prescribed mean curvature, i.e., the study of solutions of
 \be\label{eq:pmc_pde}
\begin{aligned}
 & -\Div \frac{\nabla u}{\sqrt{1+ |\nabla u|^2}} =\lambda f(u),  \ \   x \in \Omega;  \quad   u= 0, \ \   x\in \partial \Omega,
\end{aligned}
\ee
goes back to 1805 and 1806 when Thomas Young~\cite{young1805essay} and Pierre-Simon Laplace~\cite{laplace1805traite} separately looked at the properties of capillary surfaces, namely, where $f$ is linear. Later the theory, which still focused on capillary surfaces, was put on a solid mathematical foundation by Gauss~\cite{gauss1830principia} and attracted the attention of many nineteenth century scientific luminaries. During the first half of the twentieth century, the problem fell out of vogue, taking a backseat to the study of constant mean curvature surfaces (c.f.~\cite{osserman2002survey}); however in the past few decades, due to mathematical advances and the miniaturization of technology, the study of  problem \eqref{eq:pmc_pde} has come back into focus. Most modern authors have focused on the existence, nonexistence and multiplicity of positive solutions.
The case where $\Omega \subseteq \R^n$, for $n\geq 2$, has been studied by numerous authors: Concus and Finn~\cite{concus1970class,concus1974capillaryI,concus1974capillaryII,concus1975singularI,concus1975singularII,concus1976height,concus1978shape}; Giusti~\cite{giusti1976boundary,giusti1978equation,giusti1980generalized,giusti1981equilibrium}; Gilbarg and Trudinger~\cite{gilbarg1983elliptic}; Ni and Serrin~\cite{ni1985nonexistence,ni1986existence,ni1986nonexistence}; Finn~\cite{finn1986equilibrium,finn2002eight} (and the references therein); Peletier and Serrin~\cite{peletier1987ground}; Atkinson, Peletier and Serrin~\cite{atkinson1988ground,atkinson1989bounds,atkinson1992estimates}; Serrin~\cite{serrin1988positive}; Ishimura~\cite{ishimura1990nonlinear,ishimura1991generalized}; Kusano and Swanson~\cite{kusano1990radial}; Nakao~\cite{nakao1990bifurcation}; Noussair, Swanson and Jianfu~\cite{noussair1993barrier}; Bidaut-Veron~\cite{bidaut1996rotationally}; Cl{\'e}ment, Man{\'a}sevich and Mitidieri~\cite{clement1996modified}; Coffman and Ziemer~\cite{coffman1991prescribed}; Conti and Gazzola~\cite{conti2002existence}; Amster and Mariani~\cite{amster2003nonlinear}; Habets and Omari~\cite{habets2004positive}; Le~\cite{le2005some,le2008sub}; Chang and Zhang~\cite{chang2006multiple}; del Pino and Guerra~\cite{delpino2007ground}; Moulton and Pelesko~\cite{moulton2008theory,moulton2009catenoid}; Bereanu, Jebelean and Mawhin~\cite{bereanu2009radial,bereanu2010radial}; Obersnel and Omari~\cite{obersnel2009existence,obersnel2010positive}; Brubaker and Pelesko~\cite{brubaker2011nonlinear}; Brubaker and Lindsay~\cite{brubaker2011analysissingular}. Also, the case where $n=1$ has been studied by numerous authors in a recent series of papers:  Kusahara and Usami~\cite{kusahara2000barrier}; Benevieri, do \`{O} and de Medeiros~\cite{benevieri2006periodic,benevieri2007periodic}; Bonheure, Habets, Obersnel and Omari~\cite{bonheure2007classical,bonheure2007classical2}; Habets and Omari~\cite{habets2007multiple};  Obersnel~\cite{obersnel2007classical}; Pan~\cite{pan2009onedimensional}; Li and  Liu~\cite{li2010exact}; Burns and Grinfeld~\cite{burns2011steady};  Pan and Xing~\cite{pan2011timemapI,pan2011timemapII}. One of the fascinating aspects of problem \eqref{eq:pmc_pde} that these studies have revealed is the \emph{disappearing} solution behavior shown to be present for numerous choices of $f$ (cf.~\cite{pan2011timemapI},~\cite{pan2011timemapII} and the references therein). For example, Pan showed in~\cite{pan2009onedimensional} that the solution set of
\be\label{eq:1d_pmc_gb_eq}
\begin{aligned}
 & -\left(\frac{u'}{\sqrt{1+(u')^2}} \right)' =\lambda e^{u},  \ \   0<x<L;  \quad   u(0)=u(L)= 0,
\end{aligned}
\ee
which is a quasilinear analogue of the semilinear Gelfand--Bratu equation, is fully characterized by the following theorem for $\lambda >0$ (see Figure \ref{fig:BDs_bratu} for the resulting bifurcation diagram). 

\begin{nonumbtheorem}[Pan~\cite{pan2009onedimensional}]
Assume that $\lambda >0$.  
\begin{enumerate}[\normalfont(i) ]
  \item If $L < \pi$, there exists constants $\lambda_*$ and $\lambda^*$ such that \begin{inparaenum}[\normalfont(a)] \item for any $\lambda \in (0,\lambda_*)\cup \{\lambda^*\}$ there is a unique positive solution of \eqref{eq:1d_pmc_gb_eq}; \item for any $\lambda \in [\lambda_*,\lambda^*)$, there exists exactly two positive solutions of  \eqref{eq:1d_pmc_gb_eq}; \item for any $\lambda \in (\lambda^*,\infty)$, no positive solutions of  \eqref{eq:1d_pmc_gb_eq} exist. \end{inparaenum}
  
  \item If $L \geq \pi$, there exists a constant $\lambda^*$ such that \begin{inparaenum}[\normalfont(a)] \item for any $\lambda \in (0,\lambda^*)$ there exists exactly two positive solutions of  \eqref{eq:1d_pmc_gb_eq}; \item for  $\lambda = \lambda^*$, there is a unique positive solution of \eqref{eq:1d_pmc_gb_eq}; \item for any $\lambda \in (\lambda^*,\infty)$, no positive solutions of  \eqref{eq:1d_pmc_gb_eq} exist. \end{inparaenum}
 \end{enumerate}
\end{nonumbtheorem} 

\begin{figure}[!h]
 \centering
 \subfigure[$L < \pi$]{\label{fig:Llpi}\includegraphics[width=0.45\textwidth]{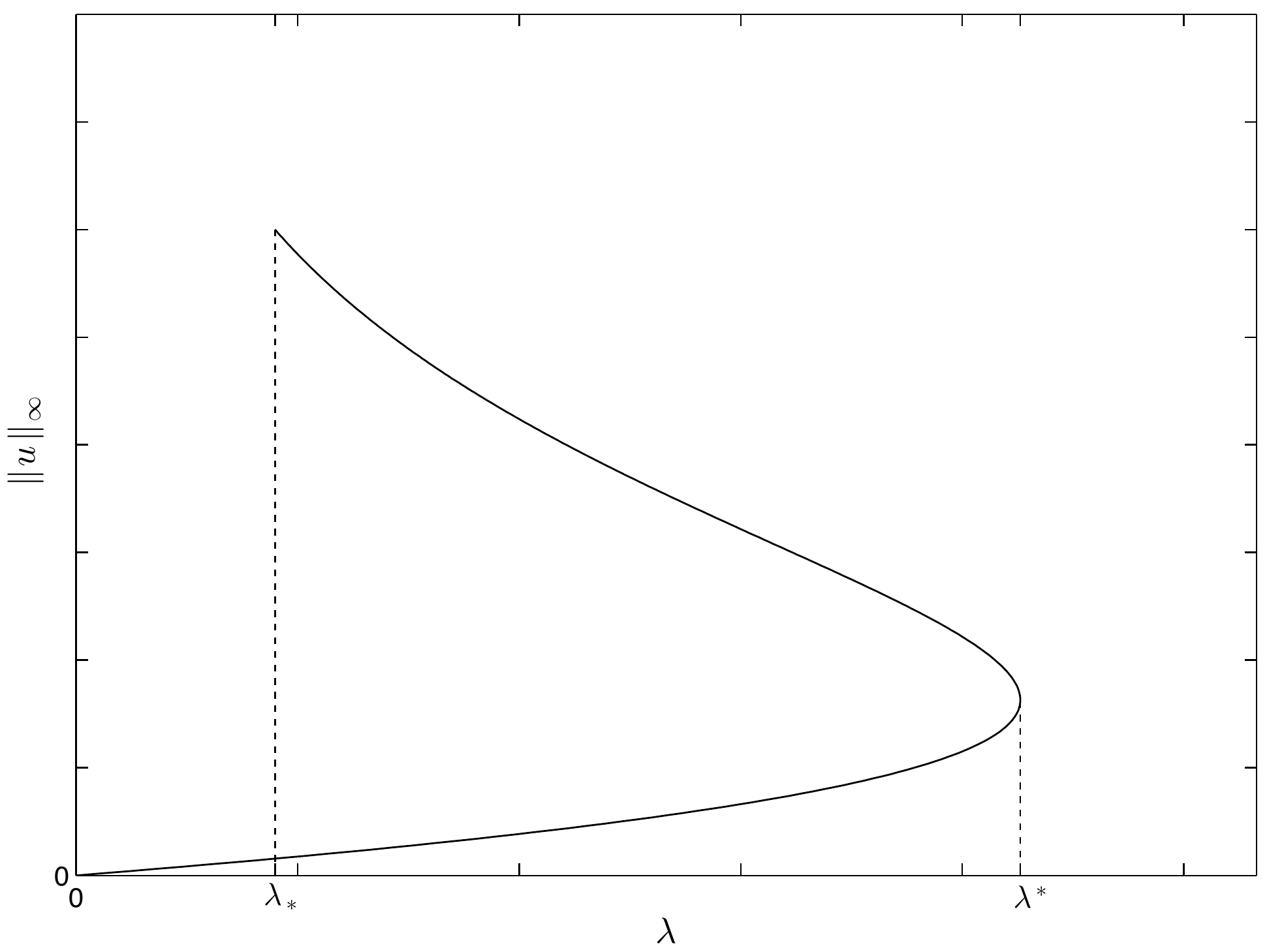}} \qquad
 \subfigure[$L \geq \pi^*$]{\label{fig:Lgeqpi}\includegraphics[width=0.45\textwidth]{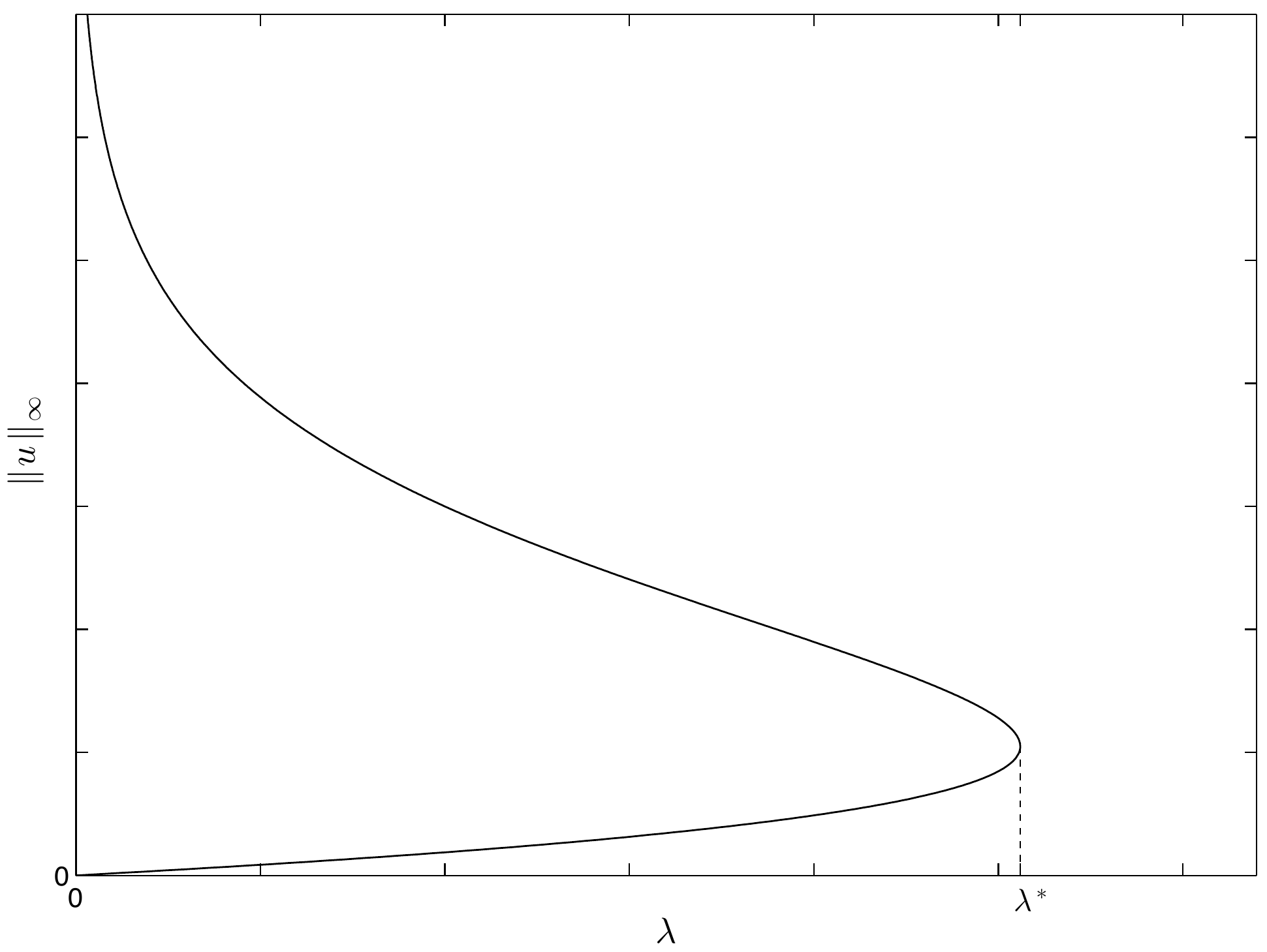}}
 \caption{Bifurcation Diagrams of \eqref{eq:1d_pmc_gb_eq} for $L< \pi$ and $L \geq \pi$. Note that in \subref{fig:Lgeqpi}, the bifurcation diagram continues; in particular, $\lambda \to 0^+$ as $\|u\|_\infty \to \infty$~\cite{pan2009onedimensional}. }\label{fig:BDs_bratu}
\end{figure}
In this paper we continue to study the disappearing solution behavior of prescribed mean curvature problem \eqref{eq:pmc_pde}, with $n=1$,  by considering the case where $f$ is an inverse square nonlinearity that is singular at $u = 1$. Namely, we study the solution set of
\be\label{eq:ode}
\begin{aligned}
 & -\left(\frac{u'}{\sqrt{1+(u')^2}} \right)' = \frac{\lambda}{(1-u)^2}, \quad   -L<x<L; \qquad   u(-L)=u(L)= 0,
\end{aligned}
\ee
with $u <1$ in $[-L,L]$, for positive $\lambda$ and $L$. Previously, these types of equations (i.e., one-dimensional singular prescribed mean curvature equations) have been studied in a general context by Bonheure et al. in \cite{bonheure2007classical}. In particular, applying their results to problem \eqref{eq:ode} gives the following: 
\begin{itemize}
 \item if $\lambda>0$ is sufficiently small, then there exists at least two classical solutions of  problem \eqref{eq:ode} \cite[Theorems 3.1 and 5.1]{bonheure2007classical};
 \item there exists a $\lambda^*>0$ such that for all $\lambda>\lambda^*$, no (classical or non-classical) positive solutions of  problem \eqref{eq:ode} exist \cite[Theorem 6.1]{bonheure2007classical}.
\end{itemize}
To build on their work, we use the method of time maps to give a description of the exact number of classical solutions, i.e., $C^2(-L,L)\cap C[-L,L]$ solutions, when the parameters $\lambda$ and $L$ vary. In doing so, we prove the following theorem that fully characterizes the solution set of $\eqref{eq:ode}$.
\begin{theorem}\label{thm:maintheorem}
 Let $L>0$  and $L^*$ be defined by $L^* \colonequals \max_{\lambda>0}g(\lambda)$, where 
\begin{equation}\label{eq:Ranalytic_thm}
 g(\lambda) = 
 \begin{cases}
	\displaystyle\frac{\lambda}{(1-\lambda^2)^{3/2}}\left(\log \frac{1+\sqrt{1-\lambda^2}}{\lambda} - \sqrt{1-\lambda^2} \right),&\mbox{for } \lambda \in (0,1), \\
	1/3, & \mbox{for } \lambda=1,\\
	 \displaystyle\frac{\lambda}{(\lambda^2-1)^{3/2}}\left(\sqrt{\lambda^2-1} - \sec^{-1}(\lambda)\right),&\mbox{for } \lambda \in (1,\infty).
	\end{cases}
\end{equation} 
\begin{enumerate}[\normalfont(i) ]
 \item\label{mthm:case2} If $L < L^*$, then there exists three values $\lambda_*$, $\lambda_{**}$ and  $\lambda^*$, which depend on $L$, such that
  \begin{enumerate}[\normalfont(a) ]
   \item for $\lambda \in (0,\lambda_*] \cup [\lambda_{**},\lambda^*)$, \eqref{eq:ode} has exactly two positive solutions;
   \item for $ \lambda \in (\lambda_*,\lambda_{**})\cup\{\lambda^*\}$, \eqref{eq:ode} has exactly one positive solution;
   \item for $\lambda > \lambda^*$, \eqref{eq:ode} has no solutions.
  \end{enumerate}
 \item\label{mthm:case1} If $L \geq L^*$, then there exists a value $\lambda^*$, which depend on $L$, such that
  \begin{enumerate}[\normalfont(a) ]
   \item for $\lambda \in (0,\lambda^*)$, \eqref{eq:ode} has exactly two positive solutions;
   \item for $\lambda = \lambda^*$, \eqref{eq:ode} has exactly one positive solution;
   \item for $\lambda > \lambda^*$, \eqref{eq:ode} has no solutions.
  \end{enumerate}
\end{enumerate}
Furthermore, in both cases, $\lambda^* <  \min \left\{ L^{-1}, \pi^2/(27L^2) \right\}.$
\end{theorem}
\begin{remark}
 Using \eqref{eq:Ranalytic_thm}, the value of $L^*$ can easily be approximated: $L^* \approx 0.3499676$.
\end{remark}
The results of this theorem are illustrated in Figure \ref{fig:BDs} for the two separate cases. Consequently, we see that the solutions set of \eqref{eq:ode} contains two bifurcations: 
\begin{enumerate}
  \item a saddle node bifurcation occurs, for all $L$, at the point $(\lambda^*(L),\|u(\cdot \: ;\lambda^*,L)\|_\infty)$, where $u(x;\lambda,L)$ is the solution of \eqref{eq:ode}; 
  \item a \emph{splitting} bifurcation occurs at $L=L^*$. That is, when $L \geq L^*$, the upper solution branch of the bifurcation diagram of \eqref{eq:ode} is continuous (see Figure \ref{fig:Lpt6}); however, when $L < L^*$, the upper solution branch splits into two parts (see Figure \ref{fig:Lpt3}).
\end{enumerate}
\begin{figure}[!h]
 \centering
  \subfigure[$L<L^*$]{\label{fig:ss1}\includegraphics[width=0.45\textwidth]{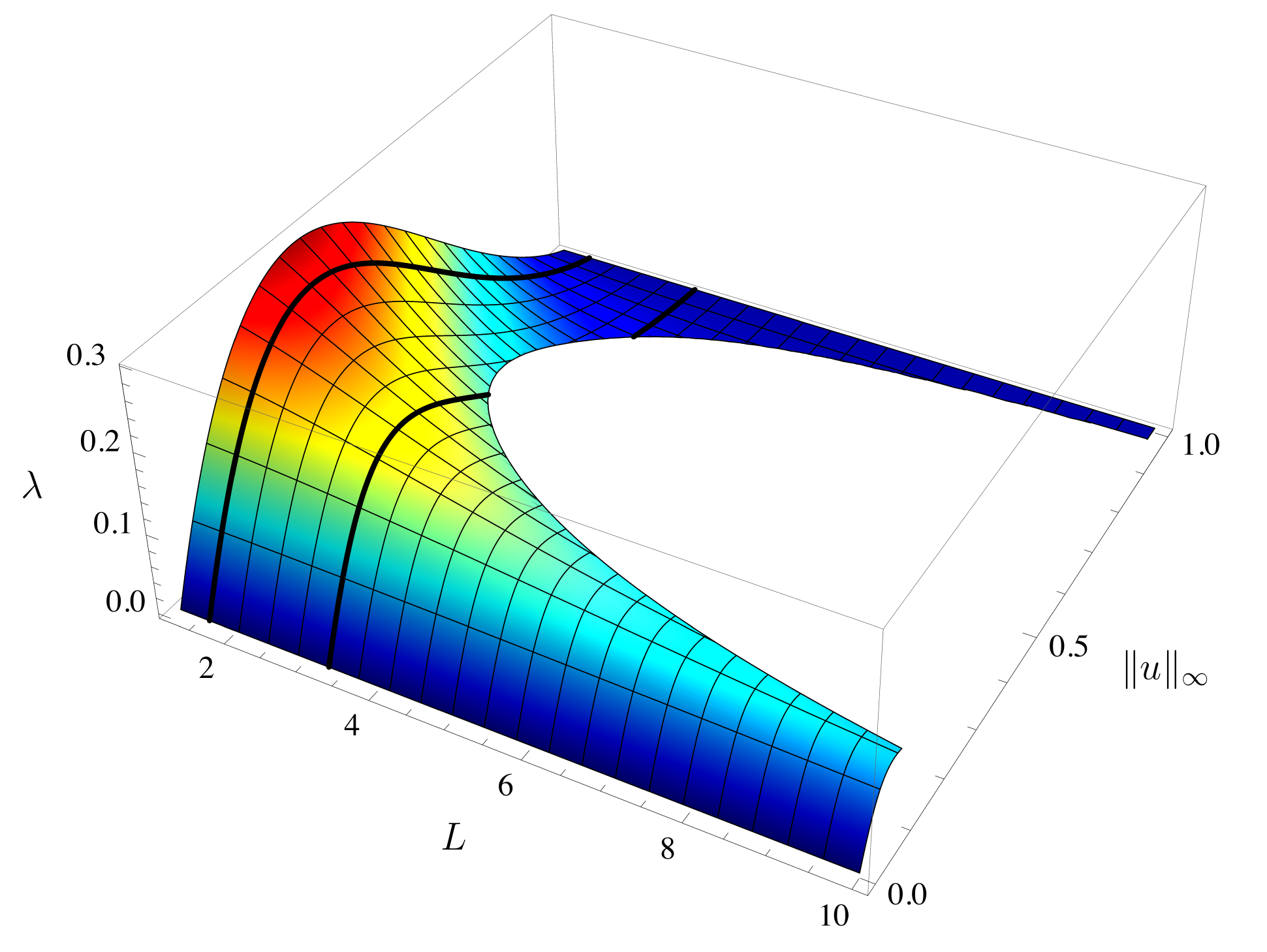}} \qquad
  \subfigure[$L \geq L^*$]{\label{fig:ss2}\includegraphics[width=0.45\textwidth]{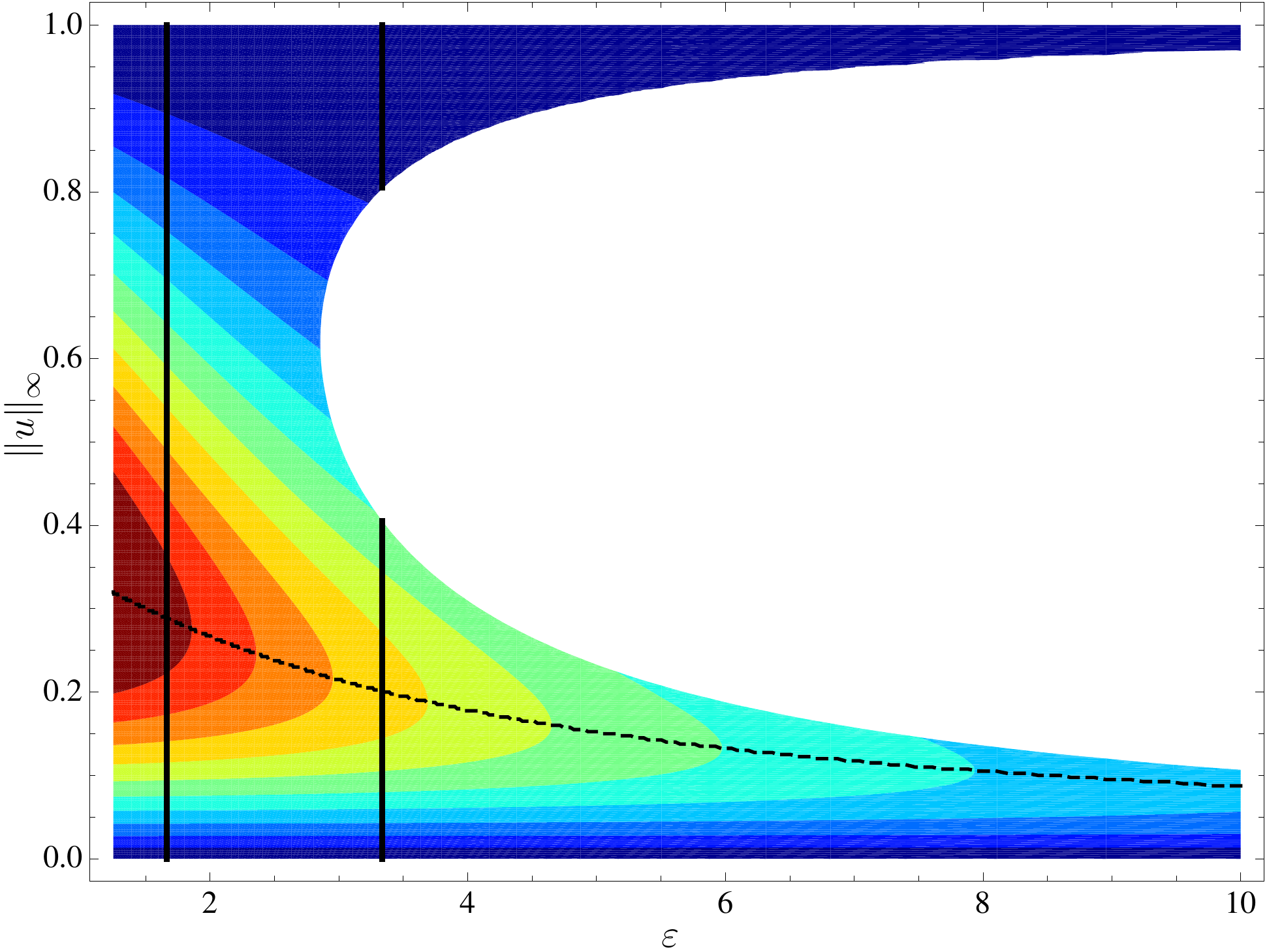}}
  \subfigure[$L<L^*$]{\label{fig:Lpt3}\includegraphics[width=0.45\textwidth]{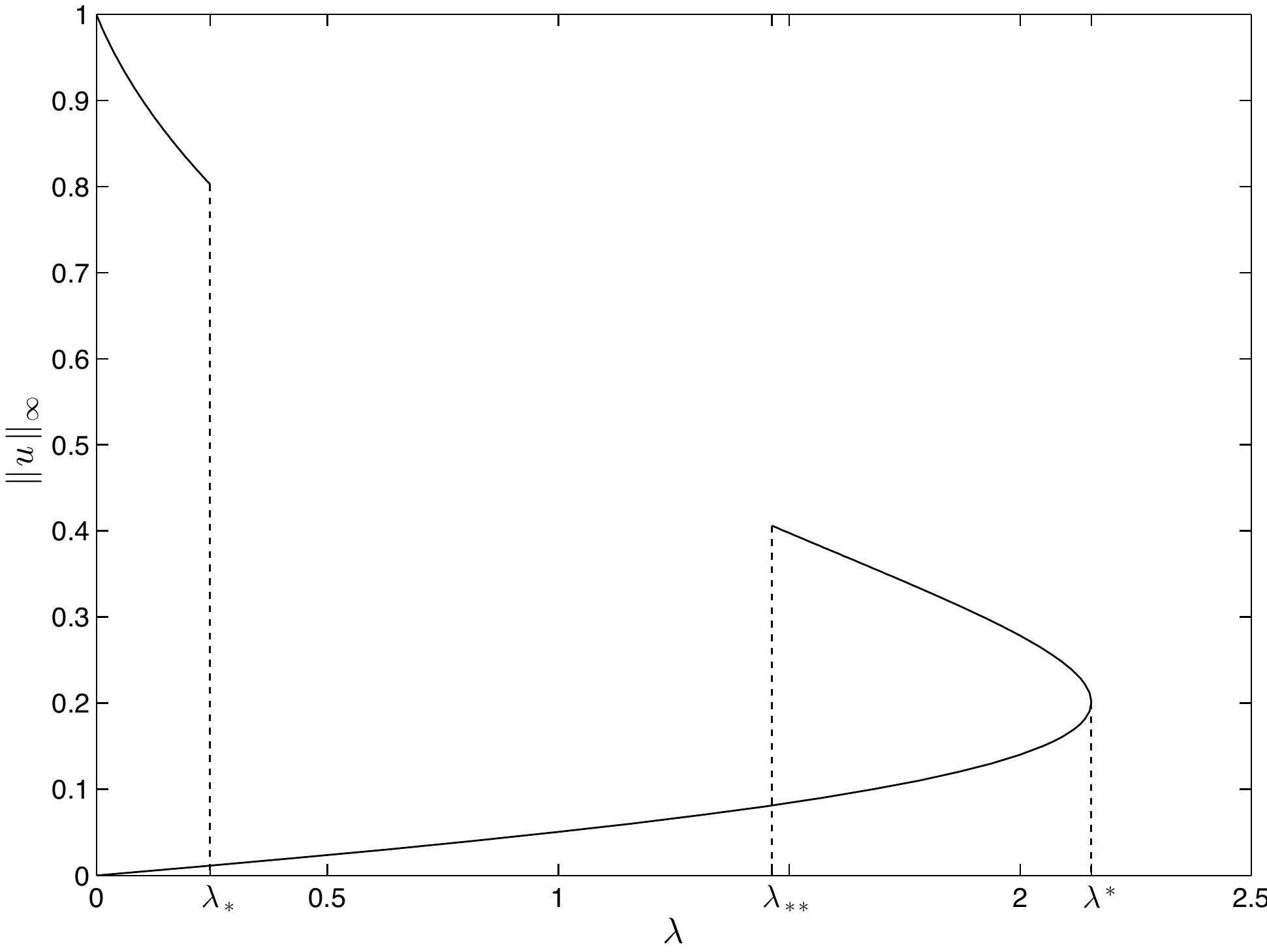}} \qquad
  \subfigure[$L \geq L^*$]{\label{fig:Lpt6}\includegraphics[width=0.45\textwidth]{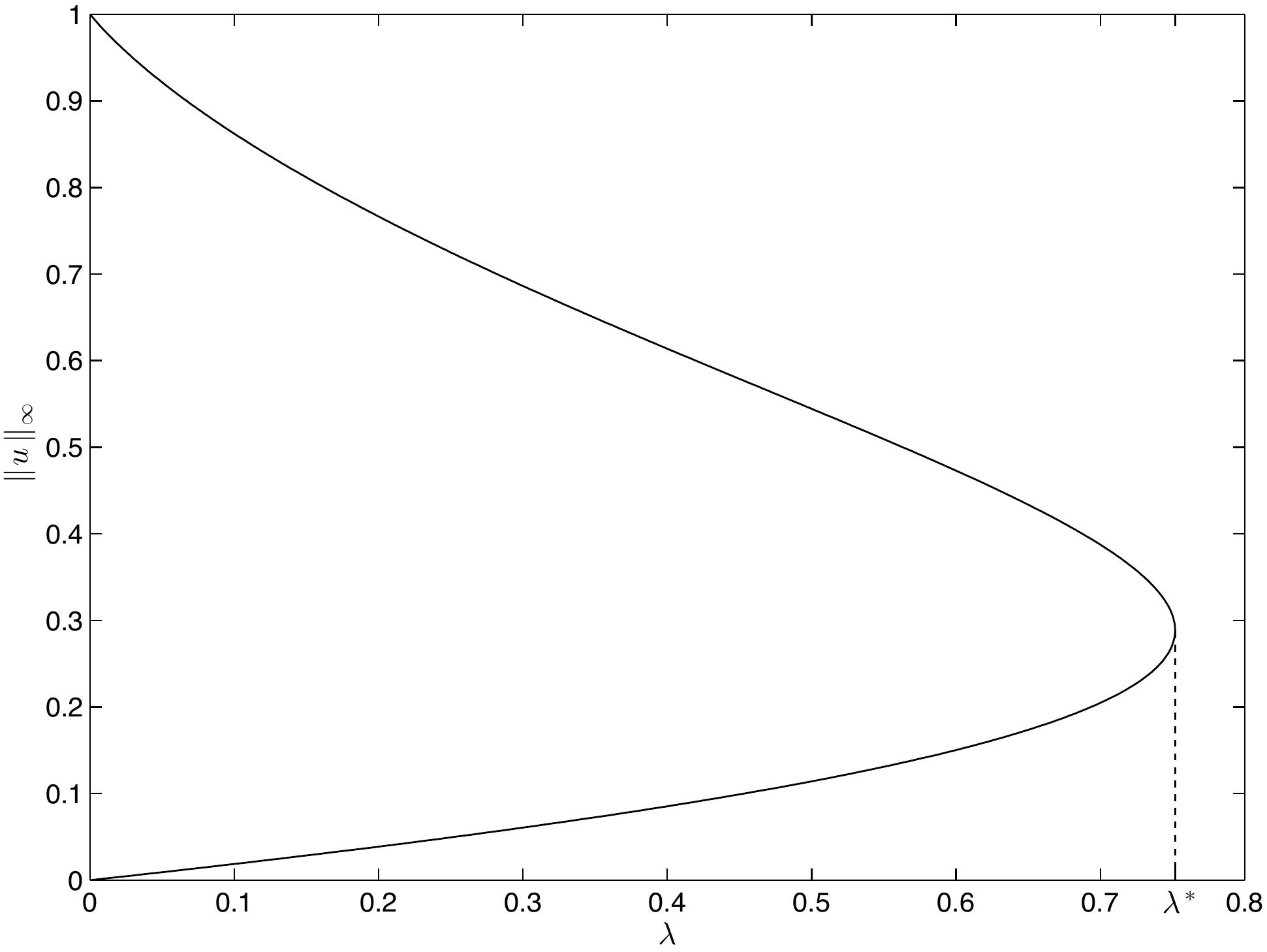}} \\
 \caption{\subref{fig:ss1} Surface where the points on it correspond to the solutions  of \eqref{eq:ode}. Here, the black contours represent solutions for fixed $L$, which yield bifurcation curves that capture the qualitative shape described in Theorem \ref{thm:maintheorem}; namely, for $L=3/10$ and $L=6/10$, which coincide with case \eqref{mthm:case2} and  case \eqref{mthm:case1}, respectively (see Subfigures \subref{fig:Lpt3} and \subref{fig:Lpt6}). \subref{fig:ss2} Contour plot of the surface given in Subfigure \subref{fig:ss2}. The new lines\emdash (thick, dashed), (thick) and (dashed)\emdash are the locations of $\lambda_*(L)$, $\lambda_{**}(L)$ and $\lambda^*(L)$ of Theorem \ref{thm:maintheorem}. \subref{fig:Lpt3} Bifurcation diagram of \eqref{eq:ode} for $L=3/10$; \subref{fig:Lpt6} Bifurcation diagram of \eqref{eq:ode} for $L=6/10$.}\label{fig:BDs}
\end{figure}

\begin{remark}
Equation \eqref{eq:ode} arises in the study of microelectromechanical systems, namely,~\cite{brubaker2011nonlinear} proposed that the steady state deflection, $u$, of an elastic membrane due to an electrostatic forcing satisfies
\[
  -\mathop{\mathrm{div}} \frac{\nabla u}{\sqrt{1+\eps^2|\nabla u|^2}} = \frac{\mu}{(1-u)^2},  \ \   x \in \Omega; \quad u<1, \ \   x \in \Omega; \quad   u= 0, \ \   x\in \partial \Omega,
\]
 where $0<\eps\ll 1$ is an aspect ratio of the device, $\mu$ is a positive ratio of the reference electrostatic force over the reference elastic force and $\Omega$ is a open, connected, bounded domain in $\R^n$, for $n=1,2$. Therefore, when $n=1$ (without loss of generality, we assume that the domain is centered at the origin), the change of variable $y = \eps^{-1} x$ yields \eqref{eq:ode}, where $\lambda \ce \eps^2 \mu$ and $L \ce 1/\eps$.
\end{remark}

\begin{remark} 
 The disappearing solutions behavior exhibited in case \eqref{mthm:case2} of Theorem \ref{thm:maintheorem} is due to the mean curvature operator; specifically, from~\cite{pelesko2003modeling}, we know  that for the related semilinear problem
\be\label{eq:1d_std_mems}
\begin{aligned}
 & -u'' = \frac{\lambda}{(1-u)^2},  \ \   -L<x<L;  \quad   u(-L)=u(L)= 0,
\end{aligned}
\ee
there exists a $\lambda^*>0$ such that for $\lambda \in (0,\lambda^*)$ there are exactly two solutions of \eqref{eq:1d_std_mems}, for $\lambda =\lambda^*$ there is a unique solution of \eqref{eq:1d_std_mems} and for $\lambda>\lambda^*$ there are no solutions of \eqref{eq:1d_std_mems}. So, in other words, the bifurcation diagram of \eqref{eq:1d_std_mems} depends on only a single parameter, which can easily be seen through the change of variable $y = \sqrt{\lambda} x$, and has same qualitative shape as the one shown in Figure \ref{fig:Lpt6}. Hence the disappearance of solutions is a byproduct of the mean curvature operator coupled with the singularity.  
\end{remark}

 We start the next section by proving some necessary conditions for the existence of \emph{classical} solutions of \eqref{eq:ode}. From this, we restrict the parameter space to where solutions may exist and use the inherent symmetry of \eqref{eq:ode} to reduce it to an equivalent problem. Then using these restrictions, in Section \ref{sect:tm}, we introduce and analyze a time map \`{a} la~\cite{pan2009onedimensional} and~\cite{smoller1980nondegenerate}. From this analysis, we then prove Theorem \ref{thm:maintheorem}.
 
\section{Necessary conditions for the existence of solutions}\label{sect:gp}

In this section, we derive some properties about the solutions  of \eqref{eq:ode}. Also, we show that solutions cannot exist under certain conditions.

\begin{lemma}\label{lem:nc1}
 If $u(\cdot \,;\lambda,L)$ is solution of \eqref{eq:ode} for fixed $\lambda>0$ and $L>0$, then $u>0$ in $(-L,L)$; furthermore, $u$ is strictly concave down in this region, which implies that its maximum, $\|u\|_\infty \in (0,1)$, is unique.
\end{lemma}

\begin{proof}
From carrying out the differentiation in differential equation of \eqref{eq:ode}, we obtain $u''(x)<0$  for all $x$ in $(-L,L)$ and $\lambda >0$, which gives the desired result.
\end{proof}

Next, it is easily verified that the ordinary differential equation given in \eqref{eq:ode} has the first integral
 \begin{equation}\label{eq:firstintegral1}
 \frac{1}{\sqrt{1+(u')^2}} - \frac{\lambda}{1-u}= E, 
\end{equation}
where $E$ is a conserved quantity of the system and can be determined by the maximum deflection of $u$. That is, let $c$ in $ (-L,L)$ be the unique value  such that $u(c) = \|u\|_\infty$; hence, $u'(c) = 0$ and first integral \eqref{eq:firstintegral1} gives $E = 1 - \lambda(1-\|u\|_\infty)^{-1}$, which implies 
 \begin{equation}\label{eq:firstintegral}
 \frac{1}{\sqrt{1+(u')^2}} - \frac{\lambda}{1-u}= 1 - \frac{\lambda}{1-\|u\|_\infty}.
\end{equation}
Using this we can prove the following necessary condition for classical solutions.

\begin{theorem}\label{eq:restriction_thm}
 Let $u(\cdot \,;\lambda,L)$ be a solution of  \eqref{eq:ode} for $\lambda>0$ and $L>0$. Then $\|u\|_\infty \leq 1/(1+\lambda).$
\end{theorem}

\begin{proof}
For contradiction assume that $\|u\|_\infty > (1+\lambda)^{-1}$. Hence,  
\be\label{eq:u_sup_bound_neg}
\|u\|_\infty  > \frac{1 - (1 + \lambda)\|u\|_\infty}{1 - \lambda- \|u\|_\infty} >0.
\ee
From Lemma \ref{lem:nc1}, we know that $u$ attains its unique maximum at a value, say $c$, in $(-L,L)$, i.e., $u(c) = \|u\|_\infty > u(x)$ for all $x$ in  $(-L,c)\cup(c,L)$; therefore, inequality \eqref{eq:u_sup_bound_neg} implies that there exists an $\hat{x}$ in $(-L,c)\cup(c,L)$ such that 
\[
u(\hat{x}) = \frac{1 - (1 + \lambda)\|u\|_\infty}{1 - \lambda- \|u\|_\infty}.
\]
Now, from first integral \eqref{eq:firstintegral}, we know 
\[
  u'(x;\lambda,L)^2 =-1 +  \frac{(1-u)^2(1-\|u\|_\infty)^2}{[(1-u)(1-\|u\|_\infty)- \lambda(\|u\|_\infty-u)]^2}, \qquad  x\neq \hat{x}.
\]
Since $u'$ is continuous and bounded in $(-L,L)$, we take $x \to \hat{x}$ so 
\[
  u'(\hat{x};\lambda,L)^2 = - 1 + \lim_{x \to \hat{x}} \frac{(1-u)^2(1-\|u\|_\infty)^2}{[(1-u)(1-\|u\|_\infty)- \lambda(\|u\|_\infty-u)]^2};
\]
however, the limit on the right-hand side diverges to infinity and we have a contradiction.
\end{proof}

 Now, from this restriction, we may also find the value of $x$ where the maximum deflection of $u$ happens. To do so, we first have the following proposition.
\begin{proposition}\label{prop:nocomplexderivative}
  Assume that $\lambda >0$ is fixed, $\alpha \in (0,1/(1+\lambda)]$ and $u(x)$ is a function whose range is contained in $(0,1)$. Then 
\refstepcounter{equation}\label{eq:inequalitylinearinu0}
  \be
   2 (1-u)(1-\alpha) -\lambda(\alpha-u) >0 \qquad \mbox{and} \qquad (1-u)(1-\alpha)-\lambda(\alpha-u) >0. \tag{\theequation a,b}\label{eq:inequalitylinearinu0_ab}
  \ee
\end{proposition}
\begin{proof}
By collecting $\alpha$ terms,
\[
 2 (1-u)(1-\alpha) -\lambda(\alpha-u) = (-2(1 - u) -\lambda) \alpha + 2 -2 u + \lambda u,
\]
which is linear in $\alpha$. Therefore, since $u\in(0,1)$, $-2(1 - u) -\lambda <0$ and  the function takes on its minimum at $\alpha = 1/(1+\lambda)$. Hence,
\[
 2 (1-u)(1-\alpha) -\lambda(\alpha-u)  \geq  \frac{\lambda(1-u +\lambda u)}{1+\lambda}> 0.
\]

Similarly, one can show $(1-u)(1-\alpha)-\lambda(\alpha-u)>0$.
\end{proof}

From this proposition, first integral  \eqref{eq:firstintegral} and Theorem \ref{eq:restriction_thm}, we have that the derivative of a solution $u(\cdot \,;\lambda,L)$ of \eqref{eq:ode} must satisfy 
\begin{equation}\label{eq:firstderivative1}
  u'(x;\lambda,L) = \begin{cases}
 	 \displaystyle \frac{\sqrt{\lambda(\|u\|_\infty-u)}\sqrt{2 (1-u)(1-\|u\|_\infty) -\lambda(\|u\|_\infty-u)}}{(1-u)(1-\|u\|_\infty)-\lambda(\|u\|_\infty-u)},  &x\in (-L,c], \vspace{3mm}\\
 \displaystyle - \frac{\sqrt{\lambda(\|u\|_\infty-u)}\sqrt{2 (1-u)(1-\|u\|_\infty) -\lambda(\|u\|_\infty-u)}}{(1-u)(1-\|u\|_\infty)-\lambda(\|u\|_\infty-u)},  & x\in (c,L),
   \end{cases}
\end{equation}
where $\|u\|_\infty \in (0,1/(1+\lambda)]$ and $\lambda>0$. Next, we use derivative \eqref{eq:firstderivative1} to prove that the solutions to \eqref{eq:ode} must be even.

\begin{theorem}\label{thm:even}
 If  $u(\cdot \,;\lambda,L)$ is a solution of  \eqref{eq:ode} for $\lambda>0$ and $L>0$, then $u$ is an even function in $x$; moreover, it attains its unique maximum at $x=0$, which implies that $u(0) = \|u\|_\infty$ and $u'(0)=0$.
\end{theorem}
\begin{proof}
From Lemma \ref{lem:nc1}, we know that $u$ attains its unique maximum at a value, say $x=c$, in $(-L,L)$. To show that $u$ is even, we use the fact that for every $U\in(0,u(c))$ there exists two values $x_1$ and $x_2$ in $(-L,c)$ and $(c,L)$, respectively, such that $u(x_1) = u(x_2)=U$. Thus, using \eqref{eq:firstderivative1}, we obtain
\[
\begin{aligned}
   \int_{-L}^{x_1} \ud{x} & = \int_{0}^{U}\frac{(1-u)(1-\|u\|_\infty)-\lambda(\|u\|_\infty-u)}{\sqrt{\lambda(\|u\|_\infty-u)}\sqrt{2 (1-u)(1-\|u\|_\infty) -\lambda(\|u\|_\infty-u)}} \ud{u} \\
    & = - \int_{U}^0\frac{(1-u)(1-\|u\|_\infty)-\lambda(\|u\|_\infty-u)}{\sqrt{\lambda(\|u\|_\infty-u)}\sqrt{2 (1-u)(1-\|u\|_\infty) -\lambda(\|u\|_\infty-u)}} \ud{u}  = \int_{x_2}^L\ud{x},
   \end{aligned}
\]
which implies $x_1 + L = L -x_2$; that is, $x_1 = -x_2$, proving that $u$ is an even function in $x$; moreover, since it has only  one maximum, $c=0$.
\end{proof}

We summarize the results of this section in the following theorem.
\begin{theorem}\label{thm:equivalent}
 The function $u(\cdot\:; \lambda, L)$ is a solution of 
\be\label{eq:equiv_ode}
 \begin{aligned}
 & -\left(\frac{u'}{\sqrt{1+(u')^2}} \right)' = \frac{\lambda}{(1-u)^2}, \ \   0<x<L; \quad u <1, \ \   -L<x<L \quad   u'(0)=u(L)= 0,
 \end{aligned}
\ee
if and only if its even extension to $(-L,L)$ is a solution of \eqref{eq:ode}. Also, no solutions of \eqref{eq:ode} exist, if $u(0) > 1/(1+\lambda)$. 
\end{theorem} 

\section{Time Map}\label{sect:tm}

In this section, to investigate the existence of solutions of \eqref{eq:ode}, we define and analyze the time map of 
 \be\label{eq:ode_ivp_tm}
 \begin{aligned}
 & -\left(\frac{u'}{\sqrt{1+(u')^2}} \right)' = \frac{\lambda}{(1-u)^2}, \ \   x>0; \qquad  u(0)= \alpha, \ \ u'(0)=0,
 \end{aligned}
\ee
for $\alpha \in (0,1/(1+\lambda)]$.
\begin{definition}
 For fixed $\lambda>0$, we define the time map $T(\cdot \:;\lambda): (0,1/(1+\lambda)] \to \R$ of  \eqref{eq:ode_ivp_tm} as the function that takes in $\alpha$ and returns the so-called ``time'' it takes for the solution of \eqref{eq:ode_ivp_tm} to hit the line $u=0$, i.e., it returns the smallest value $L$ such that $u(L;\lambda,\alpha) = 0$, where $u(\cdot\:;\lambda,\alpha) $ is the solution of \eqref{eq:ode_ivp_tm}.
\end{definition}
From this definition we have the following theorem.

\begin{theorem}\label{thm:equivalencethm}
For a given $L$, a certain value of $\lambda$ admits a solution of \eqref{eq:equiv_ode} if and only if  we can find an $\alpha$ in $(0,1/(1+\lambda)]$ such that $T(\alpha;\lambda) = L.$
\end{theorem}

\begin{remark}
Note  that the upper bound on $\alpha$ comes from Theorem \ref{thm:equivalent}. 
\end{remark}

Because of this equivalent formulation for finding solutions of \eqref{eq:equiv_ode}, it is beneficial to have an analytic expression for the time map. To this end, we deduce that similar to solutions of \eqref{eq:ode}, solutions of \eqref{eq:ode_ivp_tm} satisfy the first integral 
\[
 \frac{1}{\sqrt{1+(u')^2}} - \frac{\lambda}{1-u}= 1 - \frac{\lambda}{1-\alpha},
\]
which implies
\begin{equation}\label{eq:firstderivative}
  u'(x;\lambda,\alpha) = - \frac{\sqrt{\lambda(\alpha-u)}\sqrt{2 (1-u)(1-\alpha) -\lambda(\alpha-u)}}{(1-u)(1-\alpha)-\lambda(\alpha-u)}, 
\end{equation}
for $\alpha \in (0,1/(1+\lambda)]$ and $\lambda>0$. Hence, in separating variables, integrating with respect to $x$ from $0$ to $L$ and using the change of variables $u = \alpha z$, we deduce  from  \eqref{eq:firstderivative} that our time map is given by 
\begin{equation}\label{eq:timemap}
T(\alpha;\lambda) = \sqrt{\frac{\alpha}{\lambda}} \int_{0}^{1} \frac{(1-\alpha z)(1-\alpha)-\lambda \alpha(1- z)}{\sqrt{1- z}\sqrt{2 (1-\alpha z)(1-\alpha) - \lambda \alpha(1- z) }}\ud z \equalscolon  \int_{0}^{1} K(\alpha,z;\lambda)\ud z.
\end{equation}

Next we prove that $T(\cdot\,;\lambda)$ is well defined for $\alpha \in (0,1/(1+\lambda)]$ and is in $C^2(0,1/(1+\lambda))$.

\begin{lemma}\label{lem:differentiable}
  Let $\lambda>0$ be fixed. Then $T(\alpha;\lambda)$ exists for each $\alpha\in (0,1/(1+\lambda)]$. Moreover, $T(\alpha;\lambda)$ is differentiable at each $\alpha\in(0,1/(1+\lambda))$ with its derivative given by the formula
\begin{equation}\label{eq:Tderivative}
 T'(\alpha;\lambda) =   \int_{0}^{1} \pd{K}{\alpha}(\alpha,z;\lambda)\ud z.
\end{equation}
Furthermore, $T'(\alpha;\lambda)$ is differentiable at each $\alpha\in(0,1/(1+\lambda))$ with its derivative given by the formula
\begin{equation}\label{eq:Tpderivative}
 T''(\alpha;\lambda) =   \int_{0}^{1} \pdd{K}{{\alpha}}(\alpha,z;\lambda)\ud z.
\end{equation}
\end{lemma}
\begin{proof}
  First, $T$ is well defined for $\alpha \in (0,1/(1+\lambda)]$ since
\begin{equation}\label{eq:welldefinedT}
\begin{aligned}
 |T(\alpha;\lambda)| & \leq  \sqrt{\frac{\alpha}{\lambda}} \int_{0}^{1} \frac{1}{\sqrt{1- z}\sqrt{2 (1-\alpha z)(1-\alpha) - \lambda \alpha(1- z) }}\ud z \\
	& \leq  \frac{\sqrt{\alpha}\sqrt{1+\lambda}}{\lambda} \int_{0}^{1} \frac{1}{\sqrt{1- z}\sqrt{1- \alpha z}}\ud z  = \frac{\sqrt{1+\lambda}}{\lambda} \log\left(\frac{1 + \sqrt{\alpha}}{1 - \sqrt{\alpha}}\right),
\end{aligned}
\end{equation}
where from inequality (\ref{eq:inequalitylinearinu0}a) we have used the fact that
\be\label{eq:inequalityforlem34}
  \frac{1}{\sqrt{2 (1-\alpha z)(1-\alpha) - \lambda \alpha(1- z)}} \leq \frac{\sqrt{1+\lambda}}{\sqrt{\lambda}\sqrt{1- \alpha z}}
\ee
for $\alpha\in(0,1/(1+\lambda)]$, $\lambda>0$, $z\in (0,1)$.

Next, 
\[
\begin{aligned}
 \pd{K}{\alpha}(\alpha,z;\lambda) & = \frac{1}{2\sqrt{\lambda \alpha}}\frac{(1-\alpha z)(1-\alpha)-\lambda \alpha(1- z)}{\sqrt{1- z}\sqrt{2 (1-\alpha z)(1-\alpha) - \lambda \alpha(1- z) }} \\
	& \quad -\sqrt{\frac{\alpha}{\lambda}}\frac{1 +(1-2\alpha)z  +\lambda(1-z)}{\sqrt{1- z}\sqrt{2 (1-\alpha z)(1-\alpha) - \lambda \alpha(1- z) }}\\
	& \quad +\sqrt{\frac{\alpha}{\lambda}}\frac{((1-\alpha z)(1-\alpha)-\lambda \alpha(1- z))(2+\lambda(1-z) + (2-4\alpha)z)}{2\sqrt{1-z}\left(2 (1-\alpha z)(1-\alpha) - \lambda \alpha(1- z)\right)^{3/2}} \\
	& \equalscolon  H_1(\alpha,z;\lambda) - H_2(\alpha,z;\lambda) + H_3(\alpha,z;\lambda),
\end{aligned}
\]
which is defined for $\alpha\in(0,1/(1+\lambda)]$. Thus, 
\[
 \left|\pd{K}{\alpha}(\alpha,z;\lambda)\right|  \leq |H_1|+|H_2|+|H_3|,
\]
so that  for $\alpha\in[a,1/(1+\lambda)]$, $a>0$, and $z\in(0,1)$,
\begin{equation}\label{eq:dominatedderivative}
 \left|\pd{K}{\alpha}(\alpha,z;\lambda)\right|  \leq  \frac{C_1}{\sqrt{z}\sqrt{1- z}} + \frac{C_2}{\left(1- (1+\lambda)^{-1}z \right)^{3/2}\sqrt{1-z}} \in L^1(0,1),
\end{equation}
where $C_1$ and $C_2$ are positive constants independent of $z$ and $\alpha$ (see \ref{sec:appendix}, equations \eqref{eq:boundG1}--\eqref{eq:boundAppend4}). Therefore, we have by Theorem 10.39 in~\cite{apostol1974mathematical} that $T(\alpha;\lambda)$ is differentiable at each $\alpha\in(0,1/(1+\lambda))$ with its derivative given by equation \eqref{eq:Tderivative}.

Furthermore, 
\begin{equation}\label{eq:d2G0du02}
\begin{aligned}
 \pdd{K}{{\alpha}}(\alpha,z;\lambda) & = \frac{1}{{\alpha}^{3/2} \sqrt{\lambda}} \frac{a_0 + a_1 z + a_2 z^2 + a_3 z^3}{\sqrt{1-z}(2 (1-\alpha z)(1-\alpha) - \lambda \alpha(1- z))^{5/2}},
\end{aligned}
\end{equation}
which is defined for each $\alpha\in(0,1/(1+\lambda))$ and where
\begin{equation}\label{eq:d2G0du02coeff}
\begin{alignedat}{2}
   a_0 & \colonequals -1 +\alpha(1-\lambda),  & \qquad  a_1 & \colonequals \alpha(1+5\alpha - 10\alpha^2 + \lambda +2\alpha^3(2+\lambda)), \\
   a_2 & \colonequals -\alpha^3(10-17\alpha+\alpha^2(7+3\lambda)) ,  & \qquad  a_3 & \colonequals \alpha^4(4+3\alpha^2 - 2\lambda + \alpha(-7+3\lambda)).
\end{alignedat}
\end{equation}
Thus, for $\alpha\in[a,1/(1+\lambda)]$, $a>0$, and $z\in(0,1)$,
\[\begin{aligned}
 \left|\pdd{K}{{\alpha}}(\alpha,z;\lambda)\right| & \leq \frac{1}{{\alpha}^{3/2} \sqrt{\lambda}} \frac{|a_0| + |a_1|  + |a_2| + |a_3|}{\sqrt{1-z}(2 (1-\alpha z)(1-\alpha) - \lambda \alpha(1- z))^{5/2}} \\
	& \leq (|a_0| + |a_1|  + |a_2| + |a_3|)\frac{(1+\lambda)^{5/2}}{{a}^{3/2}\lambda^{3}}\frac{1}{\sqrt{1-z} (1- \alpha z)^{5/2}} \\
	& \leq  \frac{B}{\sqrt{1-z} (1- (1+\lambda)^{-1}z )^{5/2}},
\end{aligned}
\]
where $B$ is a positive constant that is independent of $z$ and $\alpha$. Since $(1-z)^{-1/2} (1- a z)^{-5/2}\in L^1(0,1)$ for $a\in(0,1)$, by Theorem 10.39 in~\cite{apostol1974mathematical} we have that $T'(\alpha;\lambda)$ is differentiable at each $\alpha\in(0,1/(1+\lambda))$ with its derivative given by  equation \eqref{eq:Tpderivative}.
\end{proof}

For fixed $\lambda >0$, we can now differentiate $T$ by taking the derivative inside the integral, which makes it is easier to analyze the behavior of $T$ on $(0,1/(1+\lambda)]$.

\begin{proposition}\label{prop:properties1}
 Let $\lambda>0$ be fixed. Then
\begin{enumerate}[\normalfont (i) ]
 \item\label{prop:item1} for all $\alpha \in (0,1/(1+\lambda)]$, $T(\alpha;\lambda) >0;$
 \item\label{prop:item2} $T(\alpha;\lambda) \to  0$  as $\alpha \to 0^+;$
 \item\label{prop:item3} $T'(\alpha;\lambda) \to +\infty$ as $\alpha \to 0^+;$
 \item\label{prop:item4} $\displaystyle \lim_{\alpha \to \left(\frac{1}{1+\lambda}\right)^-} T'(\alpha;\lambda) <0;$
 \item\label{prop:item5} there exists a value $\alpha^* \in (0,1/(1+\lambda)]$ such that  $ T(\alpha^*;\lambda) = \max\left\{T(\alpha;\lambda) : \alpha \in (0,1/(1+\lambda)]\right\}.$
\end{enumerate}
\end{proposition}
\begin{proof}
\begin{asparaenum}[\normalfont(i)]
 \item This follows from definition \eqref{eq:timemap} and Proposition \ref{prop:nocomplexderivative}.
 \item Since $ \lim_{\alpha \to 0^+}\log\left(\frac{1 + \sqrt{\alpha}}{1 - \sqrt{\alpha}}\right) = 0,$ by inequality \eqref{eq:welldefinedT}, we have $T(\alpha;\lambda) \to 0$ as $\alpha \to 0^+$. 
 \item From inequality \eqref{eq:inequalityforlem34}, we have  
\begin{equation}\label{eq:intofG1to0}
\begin{aligned}
 \left| \int_0^1 H_2(\alpha,z;\lambda)  \ud z\right|& \leq \int_0^1  \left|H_2(\alpha,z;\lambda) \right| \ud z \\
	& \leq (2 +\lambda)\frac{\sqrt{\alpha}\sqrt{1+\lambda}}{\lambda^2} \int_0^1 \frac{1}{\sqrt{1- z}} \frac{1}{\sqrt{1- \alpha z}}\ud z \\
	& = \frac{(2 +\lambda)\sqrt{1+\lambda}}{\lambda^2}  \log\left(\frac{1 + \sqrt{\alpha}}{1 - \sqrt{\alpha}}\right) \to 0 \quad \mbox{ as } \alpha \to 0^+,
\end{aligned}
\end{equation}
which implies that $\int_0^1 H_2(\alpha,z;\lambda) \ud z \to 0$ as $\alpha \to 0^+$. Similarly, from inequality \eqref{eq:boundG3},
\begin{equation}\label{eq:intofG2to0}
\begin{aligned}
 \left| \int_0^1 H_3(\alpha,z;\lambda)  \ud z\right|& \leq \int_0^1  \left|H_3(\alpha,z;\lambda) \right| \ud z \\
	& \leq    \frac{\sqrt{\alpha}(1+\lambda)^{3/2}(4+\lambda)}{2\lambda^{2}} \int_0^1\frac{1}{\left(1- \alpha z\right)^{3/2}} \frac{1}{\sqrt{1-z}}\ud z \\
	& =  \frac{(1+\lambda)^{3/2}(4+\lambda)}{\lambda^{2}}  \frac{\sqrt{\alpha}}{1-\alpha} \to 0 \quad \mbox{ as } \alpha \to 0^+,
\end{aligned}
\end{equation}
and $\int_0^1 H_3(\alpha,z;\lambda) \ud z \to 0$ as $\alpha \to 0^+$. But 
\begin{equation}\label{eq:intofG3to0}
\begin{aligned}
  \int_0^1 H_1(\alpha,z;\lambda)  \ud z & \geq  \frac{1}{2\sqrt{2\lambda \alpha}}  \int_0^1   \left((1-\alpha z)(1-\alpha)-\lambda \alpha(1- z)\right)\ud z\\ 
	& = \frac{1}{2\sqrt{2\lambda \alpha}}  \left(1 - \frac{3}{2}\alpha + \frac{\alpha^2}{2} - \frac{\lambda \alpha}{2}\right) \to +\infty \quad \mbox{ as } \alpha \to 0^+.
\end{aligned}
\end{equation}
Therefore, from equation \eqref{eq:Tderivative} and inequalities \eqref{eq:intofG1to0}--\eqref{eq:intofG3to0}, $ \lim_{\alpha \to 0^+} T'(\alpha;\lambda) = +\infty.$
 \item From inequality \eqref{eq:dominatedderivative} and Theorem 10.38 of~\cite{apostol1974mathematical}, we have 
 \begin{equation}\label{eq:LDClimitinside}
  \lim_{\alpha \to \left(\frac{1}{1+\lambda}\right)^-} T'(\alpha;\lambda) = \int_0^1 \pd{K}{\alpha}\left(\frac{1}{1+\lambda},z;\lambda\right) \ud z.
 \end{equation} 
 Now,
\[
 \pd{K}{\alpha}\left(\frac{1}{1+\lambda},z;\lambda\right) =\frac{3\lambda (z-1)-(z-1)^2 -\lambda^2(z^2-2 z + 3)+\lambda^3(z^2+z -1)}{\lambda\sqrt{1+\lambda}\sqrt{1-z}(1 -z + \lambda(1+z))^{3/2}}.
\]

\begin{asparaenum}[\textsc{Case} I:]
\item If $\lambda=1$, then integrating by parts gives
 \begin{equation}\label{eq:int01dGdu0eq1}
  \int_0^1 \pd{K}{\alpha}\left(\frac{1}{1+\lambda},z;\lambda\right) \ud z = \int_0^1 \frac{-8+8z-z^2}{4\sqrt{1-z}} \ud z = -\frac{8}{5}, \qquad \mbox{for }\lambda= 1.
 \end{equation} 
 
 \item If $\lambda\neq1$, then 
  \begin{equation}\label{eq:integration1dGdu0}
  \int_0^1 \pd{K}{\alpha}\left(\frac{1}{1+\lambda},z;\lambda\right) \ud z  =  \frac{1+\lambda^2-2\lambda^4+6 \lambda^2\sqrt{\lambda^2-1} \arctan\sqrt{\frac{\lambda-1}{1+\lambda}}}{\lambda(1+\lambda)(\lambda-1)^3}.
 \end{equation}
 \begin{asparaenum}[\qquad (a)]
\item If $\lambda<1$, we have $\sqrt{\lambda^2-1} = i \sqrt{1-\lambda^2}$ and 
\[
 \arctan\sqrt{\frac{\lambda-1}{1+\lambda}} = \arctan i\sqrt{\frac{1-\lambda}{1+\lambda}}  = \frac{i}{2}\log \left( \frac{\sqrt{1+\lambda}+\sqrt{1-\lambda}}{\sqrt{1+\lambda}-\sqrt{1-\lambda}}\right),
\]
which implies from  equation \eqref{eq:integration1dGdu0}  that
 \begin{equation}\label{eq:int01dGdu0eq2}
  \begin{aligned}
  \int_0^1 \pd{K}{\alpha}\left(\frac{1}{1+\lambda},z;\lambda\right) \ud z  & =  \frac{1+\lambda^2-2\lambda^4-3 \lambda^2  \sqrt{1-\lambda^2} \log \left( \frac{\sqrt{1+\lambda}+\sqrt{1-\lambda}}{\sqrt{1+\lambda}-\sqrt{1-\lambda}}\right)}{\lambda(1+\lambda)(\lambda-1)^3}<0, \quad \mbox{for }\lambda<1.
  \end{aligned}
 \end{equation}
 
\item If $\lambda>1$, we have $0<\sqrt{\lambda-1} < \sqrt{\lambda+1}$, which implies $0<\arctan\sqrt{\frac{\lambda-1}{1+\lambda}}$, and hence, from equation  \eqref{eq:integration1dGdu0},
 \begin{equation}\label{eq:int01dGdu0eq3}
  \begin{aligned}
  \int_0^1 \pd{K}{\alpha}\left(\frac{1}{1+\lambda},z;\lambda\right) \ud z  & =\frac{1+\lambda^2-2\lambda^4+6 \lambda^2\sqrt{\lambda^2-1} \arctan\sqrt{\frac{\lambda-1}{1+\lambda}}}{\lambda(1+\lambda)(\lambda-1)^3}<0, \quad \mbox{for }\lambda>1.
  \end{aligned}
 \end{equation}
 \end{asparaenum}
 \end{asparaenum}

Therefore, by  equations \eqref{eq:int01dGdu0eq1}, \eqref{eq:int01dGdu0eq2} and \eqref{eq:int01dGdu0eq3} the integral  $ \int_0^1 \partial K / \partial \alpha \left((1+\lambda)^{-1},z;\lambda\right) \ud z <0$, for $\lambda>0$, which implies from equation \eqref{eq:LDClimitinside} that
\[
  \lim_{\alpha \to \left(\frac{1}{1+\lambda}\right)^-} T'(\alpha;\lambda) <0.
\]

\item If we extend $T$ to be defined on $[0,1/(1+\lambda)]$ such that $T(0;\lambda) =0$, then by part \eqref{prop:item2} $T(\alpha;\lambda)$ is continuous on a compact set, $[0,1/(1+\lambda)]$ and it attains it supremum in $[0,1/(1+\lambda)]$. That is, by part \eqref{prop:item1}, there exists a value $\alpha^* \in (0,1/(1+\lambda)]$ such that  $ T(\alpha^*;\lambda) = \max_{\alpha \in (0,1/(1+\lambda)]} T(\alpha;\lambda) $.
\end{asparaenum}
\end{proof}

A consequence of this lemma is that there exists at least one critical point of $T(\cdot\,;\lambda)$ in $(0,1/(1+\lambda)]$. Next, we prove there exists only one.

\begin{proposition}\label{prop:onecritpt}
 Let $\lambda>0$ be a fixed value. Then there is exactly one critical point of $T$ in $(0,1/(1+\lambda)]$.
\end{proposition}
\begin{proof}
 By Proposition \ref{prop:properties1}, we know there exists at least one critical point of $T$ in $(0,1/(1+\lambda)]$. Thus, to complete the proof we only need to prove that there exists at most one critical point of $T$ in $(0,1/(1+\lambda)]$. To do so, we will show that $T''(\alpha;\lambda) <0$ for all $\alpha \in (0,1/(1+\lambda))$. 

From Proposition \ref{prop:nocomplexderivative} and  equation \eqref{eq:d2G0du02} we see that the sign of $\partial^2{K}/\partial{\alpha}^2$ for  $\alpha \in (0,1/(1+\lambda))$ depends on the numerator 
\[
 p(z,\alpha;\lambda)\colonequals a_0(\alpha;\lambda) + a_1(\alpha;\lambda)z +a_2(\alpha;\lambda) z^2 +a_3(\alpha;\lambda) z^3,
\]
where the coefficients are defined in \eqref{eq:d2G0du02coeff}. Therefore, in rearranging we have  
\[
  p(z,\alpha;\lambda) = b_0 + b_1 \lambda, 
\]
where 
\[
\begin{alignedat}{2}
   b_0 & \colonequals  -1 + 5 \alpha^2 z + 3 \alpha^6 z^3 + \alpha (1 +z) -10\alpha^3 (1  + z)z - 7\alpha^5 (1 + z)z^2 +  \alpha^4 (1+4z)(z+4)z,\\
   b_1 & \colonequals -\alpha(1 - z) - 3\alpha^5 (1-z) z^2 + 2\alpha^4(1-z^2)z.
\end{alignedat}
\]
After a long computation it can be shown that $b_0\leq0$ and $b_1 \leq 0$ for all $\alpha$ and $z$ in $(0,1)$, which implies that $p(z,\alpha;\lambda)\leq0$ for all $\alpha$ and $z$ in $(0,1)$. Therefore, $\partial^2{K}/\partial{\alpha}^2<0$ for $(\alpha,z) \in (0,1) \times (0,1)$ which implies that $\partial^2{K}/\partial{\alpha}^2<0$  for $(\alpha,z) \in (0,1/(1+\lambda)) \times (0,1) $. Thus, by equation \eqref{eq:Tpderivative}, $T''(\alpha;\lambda) <0$ for all $\alpha \in (0,1/(1+\lambda))$.
\end{proof}

From Proposition \ref{prop:properties1} and \ref{prop:onecritpt}, we now have that the graph of $T(\alpha;\lambda)$ looks like that of Figure \ref{fig:timemap}. Upon inspection, we see that the number of solutions of $T(\alpha;\lambda)=L$ depends on two values:  the value of $T$ at the end point $\alpha = 1/(1+\lambda)$ and the maximum value of $T$ for $\alpha \in (0,1/(1+\lambda)]$. With this as motivation we define the functions
\begin{alignat}{2}
	M(\lambda) & \colonequals \max \left\{ T(\alpha;\lambda) : \alpha \in (0,1/(1+\lambda)] \right\} ,  \qquad && \mbox{for } \lambda \in (0,\infty),\label{eq:timemapmaxvalue}\\
 	g(\lambda)  & \colonequals T(1/(1+\lambda);\lambda), \qquad & & \mbox{for } \lambda \in (0,\infty), \label{eq:timemaprightendvalue}
\end{alignat}
and examine their properties.

\begin{figure}[!h]
 \centering
 \includegraphics[width=.6\textwidth]{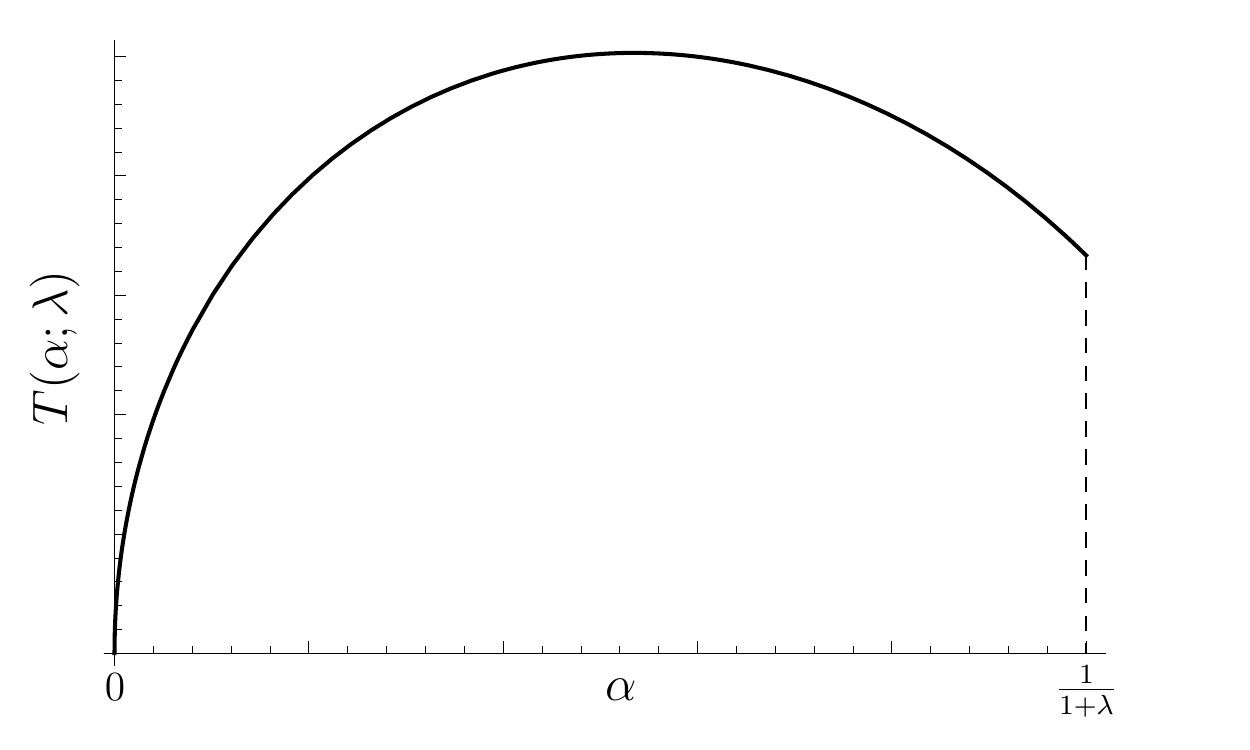}
 \caption{Plot of the time map $T(\alpha;\lambda)$.}\label{fig:timemap}
\end{figure}

\begin{lemma}\label{lem:Rdiff}
   Let $g(\lambda)$ be defined as in \eqref{eq:timemaprightendvalue}. Then $g'(\lambda)$ and $g''(\lambda)$ exist for each $\lambda >0$. Moreover, $g'(\lambda)$  and $g''(\lambda)$ are given by 
\[
 g'(\lambda)  =   \int_{0}^{1} \pd{\phi}{\lambda}(z,\lambda)\ud z, \quad \mbox{and} \quad  g''(\lambda)  =   \int_{0}^{1} \pdd{\phi}{{\lambda}}(z,\lambda)\ud z,
\]
respectively, where 
\[
 \phi(z,\lambda)\colonequals K(1/(1+\lambda),z;\lambda) =   \frac{\lambda}{(1+\lambda)^{3/2}}\frac{z}{\sqrt{1 - z}\sqrt{1 - z + \lambda(1 + z)}}.
\]
\end{lemma}

\begin{proof}
 First, $1 - z + \lambda(1 + z) = 1 + \lambda + (-1 + \lambda) z$, which is linear in $z$, and hence,
\be\label{eq:minforR}
 1 - z + \lambda(1 + z) \geq \min\left\{1 +\lambda,2\lambda \right\} >0.
\ee
Then, from the definition of  $g(\lambda)$, we have
\begin{equation}\label{eq:TMendpointbound}
\begin{aligned}
  0< g(\lambda)  & = \frac{\lambda}{(1+\lambda)^{3/2}} \int_0^1 \frac{z}{\sqrt{1 - z}\sqrt{1 - z + \lambda(1 + z)}} \ud z \\
	& \leq  \frac{\lambda}{(1+\lambda)^{3/2}}\frac{1}{\sqrt{\min\left\{1 +\lambda,2\lambda \right\}}} \int_0^1 \frac{1}{\sqrt{1 - z}} \ud z < \infty,
\end{aligned}
\end{equation}
for each $\lambda >0$. Moreover, for $\lambda>0$,
\[
    \pd{\phi}{\lambda}=  -\frac{(\lambda^2-1) z + ( \lambda^2 - \lambda +1) z^2}{(1 + \lambda)^{5/2} \sqrt{ 1 - z} (1 - z + \lambda(1 + z))^{3/2}}
\]
and
\[
    \pdd{\phi}{\lambda} =  \frac{( 2 \lambda^3- 6 \lambda -4 ) z + (4 \lambda^3- 4 \lambda^2 - 2 \lambda  + 6) z^2 + (2 \lambda^3- 4 \lambda^2+ 7 \lambda-2) z^3}{(1 + \lambda)^{7/2} \sqrt{ 1 - z} (1 - z + \lambda(1 + z))^{5/2}}.
\]
 Thus, from inequality \eqref{eq:minforR}, we obtain for each $\lambda$ in $[a,b]$, where $0<a < b$,
\[
   \left|\pd{\phi}{\lambda}\right| \leq   \frac{2 b^2 + b +2}{(1 + a)^{5/2}(\min\left\{1 +a,2a \right\})^{3/2}}\frac{1}{ \sqrt{ 1 - z}} \leq \frac{C_1}{\sqrt{1-z}}
\]
and
\[
   \left|\pdd{\phi}{\lambda}\right| \leq   \frac{8 b^3 + 8 b^2 + 15 b+ 12 }{(1 + a)^{7/2} (\min\left\{1 +a,2a \right\})^{5/2}}\frac{1}{ \sqrt{ 1 - z}}\leq \frac{C_2}{\sqrt{1-z}}.
\]
Here $C_1$ and $C_2$ are positive constants independent of $\lambda$ and $z$. Since $(1-z)^{-1/2} \in L^1(0,1)$, by Theorem 10.39 in~\cite{apostol1974mathematical} we have our result.
\end{proof}

To simplify $g$, we have from the definition 
\[
    g(\lambda)   = \frac{\lambda}{(1+\lambda)^{3/2}}\int_{0}^{1}  \frac{z}{\sqrt{1 - z}\sqrt{1 - z + \lambda(1 + z)}}\ud z.
\]
Thus,
\begin{equation}\label{eq:Rat1}
\begin{aligned}
  g(1)  & = \frac{1}{4}\int_{0}^{1}  \frac{z}{\sqrt{1 - z}}\ud z = \frac{1}{3}.
\end{aligned}
\end{equation}
For $\lambda \neq 1$, the function 
\begin{equation}\label{eq:Rmunot1}
\begin{aligned}
    g(\lambda)  & = \frac{\lambda\sqrt{1-\lambda}}{(1+\lambda)^{3/2}}\int_{0}^{1}  \frac{z}{\sqrt{((1 - \lambda) z - 1)^2 - \lambda^2}}\ud z \\
    	& = -\frac{\lambda}{(1-\lambda^2)^{3/2}}\int_{ 1}^{\lambda}  \frac{-s + 1}{\sqrt{s^2 - \lambda^2}}\ud s \\ 
	& = -\frac{\lambda}{(1-\lambda^2)^{3/2}}\begin{cases}
	 \int_{ \lambda}^{1}  (s - 1)(s^2 - \lambda^2)^{-1/2} \ud s,&\mbox{for } \lambda \in (0,1),  \vspace{1mm} \\
	  i^{-1}\int_{ 1}^{\lambda}  (1-s)(\lambda^2-s^2)^{-1/2}\ud s,&\mbox{for } \lambda\in(1,\infty)
	\end{cases}\\
	& = \begin{cases}
	- \lambda(1-\lambda^2)^{-3/2}\left(\sqrt{1-\lambda^2} - \log[(1+\sqrt{1-\lambda^2})/\lambda]\right),&\mbox{for } 0<\lambda<1, \vspace{1mm}\\
	 \lambda  (\lambda^2-1)^{-3/2} \left(\sqrt{\lambda^2-1} - \sec^{-1}(\lambda)\right),&\mbox{for } 1<\lambda<\infty,
	\end{cases}
\end{aligned}
\end{equation}
and we have by equations \eqref{eq:Rat1}--\eqref{eq:Rmunot1} that
\begin{equation}\label{eq:Ranalytic}
 g(\lambda) = 
 \begin{cases}
	\displaystyle \lambda(1-\lambda^2)^{-3/2}\left(\log[(1+\sqrt{1-\lambda^2})/\lambda] - \sqrt{1-\lambda^2} \right),&\mbox{for } 0<\lambda<1, \\
	1/3, & \mbox{for } \lambda=1,\\
	 \displaystyle \lambda  (\lambda^2-1)^{-3/2}\left(\sqrt{\lambda^2-1} - \mathrm{arcsec}\ {\lambda}\right),&\mbox{for } 1<\lambda<\infty.
	\end{cases}
\end{equation} 
Now, we may prove the following properties of $g(\lambda)$.
\begin{proposition}\label{prop:Rproperties}
 Let $g(\lambda)$ be defined as in \eqref{eq:timemaprightendvalue}. Then
\begin{enumerate}[\normalfont(i) ]
 \item $g(\lambda) > 0$ for all $\lambda$.
 \item\label{item:Rto0asmuto0} $g(\lambda) \to 0$ as $\lambda \to 0^+$;
 \item\label{item:Rto0asmutoinfty} $g(\lambda) \to 0$ as $\lambda \to +\infty$;
 \item\label{item:Rprimetoinftyasmuto0}  $g'(\lambda) \to +\infty$ as $\lambda \to 0^+$;
  \item\label{item:Rprimetoinftyasmutoinfty}  $g'(\lambda) \to 0^-$ as $\lambda \to +\infty$;
 \item\label{item:alphamax} there exists exactly one critical point $c$ of $g$ for $\lambda>0$. Moreover, $c \in (0,1)$ and  $g(c) = L^*<\infty$, where
 \begin{equation}\label{eq:Rmax}
  L^* \colonequals \max_{\lambda>0}g(\lambda).
  \end{equation}
\end{enumerate}
\end{proposition}
\begin{proof}
\begin{enumerate}[\normalfont(i) ]
 \item Follows from inequality \eqref{eq:TMendpointbound}.
 \item From inequality \eqref{eq:TMendpointbound}, we see  $g(\lambda) \to 0$ as $\lambda \to 0^+$.
 
 \item From inequality \eqref{eq:TMendpointbound}, we see  $g(\lambda) \to 0$ as $\lambda \to +\infty $.
 
 \item From equation \eqref{eq:Ranalytic},
\[
  g'(\lambda)  =  \frac{-(2 + \lambda^2) (1 - \lambda^2 + \sqrt{1 - \lambda^2}) + (1 + 2 \lambda^2) (1 + \sqrt{1 - \lambda^2}) \log[(1+\sqrt{1-\lambda^2})/\lambda]}{(1 - \lambda^2)^{5/2} (1 +\sqrt{1 - \lambda^2})},
\]
for $\lambda\in(0,1)$, which, because $\log[(1+\sqrt{1-\lambda^2})/\lambda] \to+\infty$ as $\lambda\to 0^+$, implies $g'(\lambda)\to+\infty$ as $\lambda\to 0^+$.

\item From equation \eqref{eq:Ranalytic},
\[
  g'(\lambda)  =  \frac{(1 + 2 \lambda^2) \mathrm{arcsec}\ \lambda - \sqrt{\lambda^2-1} (2 + \lambda^2)}{(\lambda^2-1)^{5/2}} =   - \frac{1}{\lambda^{2}} + \bigoh(\lambda^{-3}) \mbox{ as } \lambda \to +\infty,
\]
for $\lambda>1$, which implies $\lim_{\lambda \to +\infty}g'(\lambda) = 0.$

\item From equation \eqref{eq:Ranalytic}, we deduce that $g' <0$ on $(1,\infty)$. Furthermore, by Lemma \ref{lem:Rdiff},
\be\label{eq:derivativeRat1}
 g'(1) = \lim_{\lambda\to 1^+}\frac{(1 + 2 \lambda^2) \mathrm{arcsec}\ \lambda - \sqrt{\lambda^2-1} (2 + \lambda^2)}{(\lambda^2-1)^{5/2}} = -\frac{1}{15}<0,
\ee
and we have  $g' <0$ on $[1,\infty)$. Therefore, $g'$ cannot have a critical point in $[1,\infty)$. Next, from part \eqref{item:Rprimetoinftyasmuto0} of this proposition, inequality \eqref{eq:derivativeRat1} and Lemma \ref{lem:Rdiff}, we have by the intermediate value theorem for derivatives that there exists a $c$ in $ (0,1)$ such that $g'(c)=0$. Then using equation \eqref{eq:Ranalytic} it can be shown that $g''(\lambda)<0$ for $\lambda\in (0,1)$, which implies that $c$ corresponds to a local max and must be unique. Also, since $g(\lambda)>0$ and $g'(\lambda)<0$ for $\lambda \in [1,\infty)$, $g(c) > g(1) \geq g(\lambda) >0$, and $c$ corresponds to a global max, i.e., $g(c) =\max_{\lambda>0}g(\lambda)$.
\end{enumerate}
\end{proof}

From this analysis we have that the graph of $g(\lambda)$ is given by Figure \ref{fig:rightendpoint}.  Furthermore, with the help of Proposition \ref{prop:Rproperties} we can numerically estimate $L^*$ and find $L^* \approx 0.3499676.$
\begin{figure}[!h]
 \centering
 \includegraphics[width=.6\textwidth]{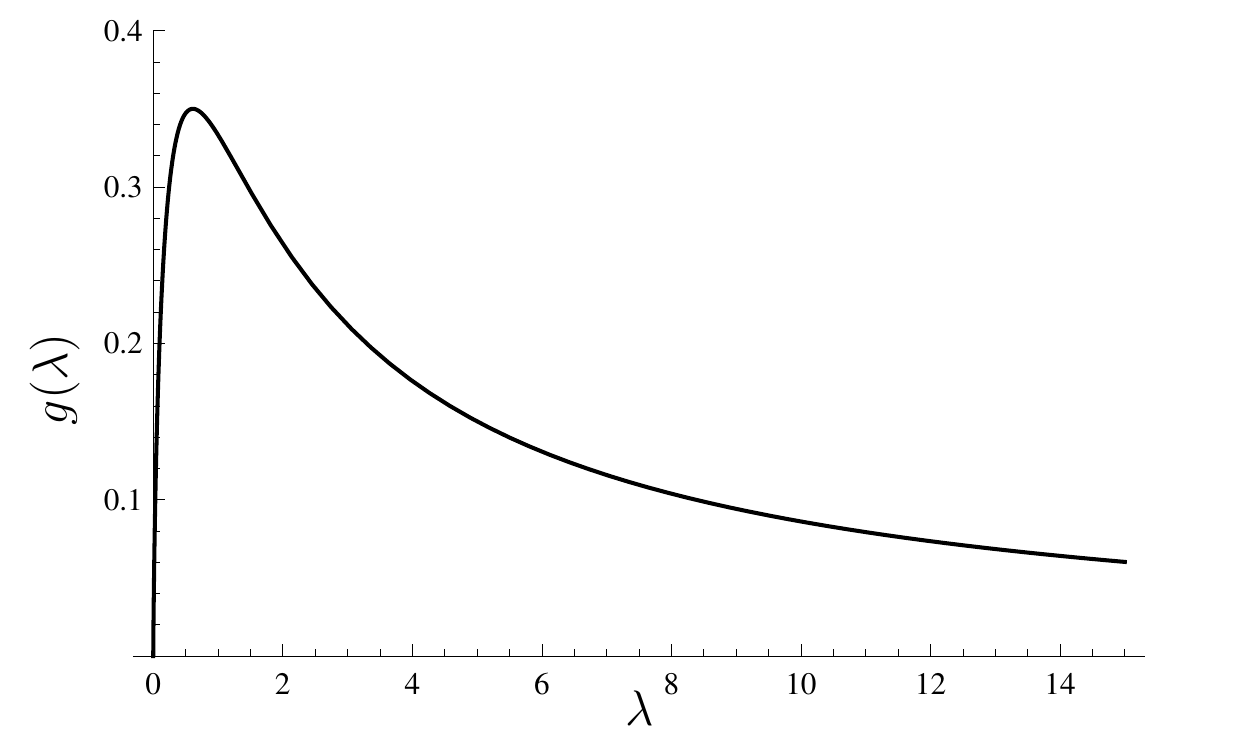}
 \caption{Graph of the right end point, $g(\lambda)$, of the time map $T(\alpha; \lambda)$.}\label{fig:rightendpoint}
\end{figure}

Next, we look at the properties of $M(\lambda)$.
\begin{proposition}\label{prop:Mproperties}
 Assume that $M(\lambda)$ is defined as in \eqref{eq:timemapmaxvalue}. Then 
\begin{enumerate}[\normalfont(i) ]
 \item\label{Mprop:item1} $M(\lambda)$ is well-defined and continuous for $\lambda>0$;
 \item\label{Mprop:item2} $M(\lambda) \to  0^+$ as $\lambda \to +\infty$;
 \item\label{Mprop:item3} $M(\lambda) \to  +\infty$ as $\lambda \to 0^+$.
\end{enumerate}
\end{proposition}
\begin{proof}
\begin{enumerate}[\normalfont(i) ]
 \item By Proposition \ref{prop:properties1}\eqref{prop:item5}, we have that $M(\lambda)$ is well defined for each $\lambda >0$. Furthermore, by the definition it is easy to see that $T(\alpha;\lambda)$ is continuous for each point in $(0,1/(1+\lambda)]\times(0,\infty)$ which implies that $M(\lambda)$ is continuous on $(0,\infty)$.
 \item For any $\lambda>0$ and $\alpha \in (0,1/(1+\lambda)]$ we have by inequality \eqref{eq:welldefinedT} and part \eqref{prop:item1} of Proposition \ref{prop:properties1} that
\[
\begin{aligned}
 0 \leq T(\alpha;\lambda) &\leq  \frac{\sqrt{\alpha}\sqrt{1+\lambda}}{\lambda} \int_{0}^{1} \frac{1}{\sqrt{1- z}}\frac{1}{\sqrt{1- (1+\lambda)^{-1} z}}\ud z  \leq \frac{2\sqrt{1 + \lambda}}{\lambda}   \log \left(\frac{1 + \sqrt{1 + \lambda}}{\sqrt\lambda}\right),
\end{aligned}
\]
which implies
\[
 0 \leq M(\lambda) \leq \frac{2\sqrt{1 + \lambda}}{\lambda}   \log \left(\frac{1 + \sqrt{1 + \lambda}}{\sqrt\lambda}\right), \quad \mbox{for} \quad \lambda>0.
\]
 Now, the right hand side has the far-field behavior  $2/\lambda + \bigoh(\lambda^{-2})$ as $\lambda \to +\infty$. Therefore, $M(\lambda) \to 0^+$ as $\lambda \to +\infty$. 

\item For any $\lambda>0$ and $\alpha \in (0,1/(1+\lambda)]$ we have 
\[
\begin{aligned}
 M(\lambda) & \geq T\left(\frac{1}{2(1+\lambda)};\lambda\right) \\
	& = \frac{1}{4 \sqrt{\lambda}(1+\lambda)^{3/2}} \int_{0}^{1} \frac{2 + 4 \lambda - z + 2 \lambda^2 (1 + z)}{\sqrt{1- z} \sqrt{2 - \lambda (-5 + z) - z + \lambda^2 (3 + z)}}\ud z \\	
 	& \geq \frac{1 + 4 \lambda + 2 \lambda^2}{2 \sqrt{\lambda} (1 + \lambda)^{3/2} \sqrt{ 2 + 5 \lambda + 4 \lambda^2}}  \to +\infty  \quad \mbox{ as } \lambda \to 0^+.
\end{aligned}
\]
\end{enumerate}
\end{proof}

Now, we can prove the following lemma:

\begin{lemma}\label{lemma:llbmt}
 Let $T$ and $L^*$ be defined as \eqref{eq:timemap} and \eqref{eq:Rmax}, respectively. Then 
 \begin{enumerate}[\normalfont(i) ]
 \item\label{item1:strictlydecreasing} for fixed $\alpha$, $T(\alpha;\lambda)$ is strictly decreasing with respect to $\lambda$, which implies that $M$ is strictly decreasing;
 \item\label{item:muunderstars} if $0<L<L^*$, then there exist two unique constants $\lambda_*$ and $\lambda_{**}$, where $\lambda_*< \lambda_{**}$, such that $g(\lambda_*) = g(\lambda_{**}) =L$;
 \item\label{item:muoverstars} for each $L$ in $(0,\infty)$, there exists a unique $\lambda^*$ such that $M(\lambda^*) =  L$;
 \item\label{item3:starstarstar} if $0<L<L^*$, then $\lambda_* < \lambda_{**} < \lambda^*$, where the values $\lambda_*, \lambda_{**}$ and $\lambda^*$ are from parts \eqref{item:muunderstars} and  \eqref{item:muoverstars}.
\end{enumerate}
\end{lemma}
\begin{proof}
 \begin{enumerate}[\normalfont(i) ]
 \item By definition \eqref{eq:timemap},
\[
 T(\alpha;\lambda) = \sqrt{\frac{\alpha}{\lambda}} \int_{0}^{1} \frac{(1-\alpha z)(1-\alpha)-\lambda \alpha(1- z)}{\sqrt{1- z}\sqrt{2 (1-\alpha z)(1-\alpha) - \lambda \alpha(1- z) }}\ud z.
\]
For fixed values of $\alpha$ and $z$, the integrand is decreasing with respect to $\lambda$, which yields the desired result.

\item From parts \eqref{item:Rto0asmuto0} and \eqref{item:alphamax} of Proposition \ref{prop:Rproperties} and $0<L<L^*=g(\alpha)$, we have by the intermediate value theorem that there exists a $\lambda_*\in(0,\alpha)$ such that $g(\lambda_*)=L$. Furthermore, by part (\ref{item:alphamax}) of Proposition \ref{prop:Rproperties}, $\alpha$ is unique and $g(\alpha)$ is a maximum. This implies that $g$ is increasing on $(0,\alpha)$, and $\lambda_*$ is unique. 

Next, from part \eqref{item:Rto0asmutoinfty} of Proposition \ref{prop:Rproperties}, there exists an $N>0$, such that $\lambda\geq N$ implies  
$L> g(\lambda)$, and hence,  $g(\lambda) \ne L$ for $\lambda \geq N$. Furthermore, since $g$ is continuous on $[\alpha,N]$ and $g(N)<L<L^*$, we have by the intermediate value theorem that there exists a $\lambda_{**} \in (\alpha,N)$ such that $g(\lambda_{**})=L$. Again using part (\ref{item:alphamax}) of Proposition \ref{prop:Rproperties}, we find that $g$ is decreasing on $(\alpha,N]$ which implies that $\lambda_{**}$ is unique. 

By construction, $\lambda_*<\alpha<\lambda_{**}$, and claim \eqref{item:muunderstars} holds.

\item From Proposition \ref{prop:Mproperties}\eqref{Mprop:item2}, we know there exists $b>0$ such that $\lambda \geq b$ implies $M(\lambda) < L$. Similarly, from  Proposition \ref{prop:Mproperties}\eqref{Mprop:item3}, we know that there exists an $a>0$ (and $a<b$), such that $\lambda\leq a$ implies $M(\lambda) > L$. Thus, since $M$ is continuous on $[a,b]$, and $M(b) <L < M(a)$, by the intermediate value theorem we have that there exists an $\lambda^*\in [a,b] \subset (0,\infty)$ such that $M(\lambda^*)=L$. 

By part \eqref{item1:strictlydecreasing}, $M(\lambda)$ is strictly decreasing and thus $\lambda^*$ is unique.

\item From part \eqref{item:muunderstars}, $\lambda_{*} < \lambda_{**}$. Furthermore, $M(\lambda^*) = L = g(\lambda_{**})<M(\lambda_{**})$ where the last inequality is from the definition of $M$ and $g$\emdash \eqref{eq:timemapmaxvalue} and \eqref{eq:timemaprightendvalue}, respectively\emdash and  Proposition \ref{prop:properties1}\eqref{prop:item4}. Thus, since $M$ is strictly decreasing, $\lambda_{**}<\lambda^*$.
\end{enumerate}
\end{proof}

Now, we are setup to prove the Theorem \ref{thm:maintheorem} which characterizes the solution set of \eqref{eq:ode}.

\begin{proof}
 Let $T(\alpha;\lambda)$ be defined by \eqref{eq:timemap}. Then, by Theorems \ref{thm:equivalent} and  \ref{thm:equivalencethm}, finding solutions to \eqref{eq:ode} is equivalent to finding $\alpha\in(0,1/(1+\lambda)]$, $\lambda>0$, such that 
 \begin{equation}\label{eq:timemapsolnmt}
  T(\alpha;\lambda) = L.
 \end{equation} Specifically, $\alpha$ is a solution of $T(\alpha;\lambda)=L$ if and only if $\alpha=u(0;\lambda,L)$, where $u(\cdot\:;\lambda,L)$ is a solution of \eqref{eq:ode}, and  the numbers of solutions of \eqref{eq:ode} and \eqref{eq:timemapsolnmt} are the same. Therefore, we look at the solutions of \eqref{eq:timemapsolnmt} in the two situations: (\ref{mthm:case2}) $L < L^*$; (\ref{mthm:case1}) $L \geq L^*$. The following analysis is illustrated in Figure \ref{fig:twocases}.

\begin{enumerate}[\normalfont(i) ]
 \item Let $L <L^*$. By Lemma \ref{lemma:llbmt}, we know there exists constants $\lambda_*,\ \lambda_{**}$ and $\lambda^*$ such that $\lambda_* < \lambda_{**} < \lambda^*$ and $g(\lambda_*) = g(\lambda_{**})=M(\lambda^*)=L$. (a) If $0 < \lambda \leq \lambda_*$, then $\lambda< \lambda^*$ which, since $M$ is strictly decreasing implies $M(\lambda) > M(\lambda^*) =L$. Furthermore, by the proof of \eqref{item:muunderstars} of Lemma \ref{lemma:llbmt}, $g$ is increasing on $(0,\lambda_*]$ and $g(\lambda)\leq g(\lambda_*)$; hence, $g(\lambda) \leq L<M(\lambda)$, and by Proposition \ref{prop:properties1} and \ref{prop:onecritpt} there exists two solutions of $T(\alpha;\lambda)=L$. (b) If $\lambda_*<\lambda<\lambda_{**}$, then by Proposition \ref{prop:Rproperties} and Lemma \ref{lemma:llbmt}, $g(\lambda) \in (L,g(\alpha)]$. Thus, $g(\lambda)> L$ and $M(\lambda)> L$, and by Propositions \ref{prop:properties1} and \ref{prop:onecritpt} there exists one solution of $T(\alpha;\lambda)=L$. (c) If $\lambda_{**} \leq \lambda < \lambda^*$, then $M(\lambda) > M(\lambda^*) =L$; also by the proof of \eqref{item:muunderstars} of Lemma \ref{lemma:llbmt}, $g$ is decreasing on $[\lambda_{**},\lambda^*)$ which implies $g(\lambda) \leq g(\lambda_{**})= L$. Hence, $g(\lambda) \leq L<M(\lambda)$, and by Proposition \ref{prop:properties1} and \ref{prop:onecritpt} there exists two solutions of $T(\alpha;\lambda)=L$. (d) If $\lambda =\lambda^*$, then $M(\lambda) =L$, and there exists only one solution of $T(\alpha;\lambda)=L$. (e) If $\lambda >\lambda^*$, then $M(\lambda) < M(\lambda^*) = L$, and $T(\alpha;\lambda) \neq L$.
 
 \item Let $L \geq L^*$ which implies that $g(\lambda) \leq L$ for $\lambda >0$. Also, by Lemma \ref{lemma:llbmt}, we know there exists a constant $\lambda^*$ such that $M(\lambda^*)=L$. (a) If $0<\lambda < \lambda^*$, then $M$ is strictly decreasing which yields $M(\lambda)>M(\lambda^*)=L$. Thus, $g(\lambda) \leq L<M(\lambda)$, and we have by Proposition \ref{prop:properties1} and \ref{prop:onecritpt} that there exists two solutions of $T(\alpha;\lambda)=L$. (b) If $\lambda=\lambda^*$, then $M(\lambda) =L$, and there exists only one solution of $T(\alpha;\lambda)=L$. (c) If $\lambda >\lambda^*$, then $M(\lambda) < M(\lambda^*) = L$, and $T(\alpha;\lambda) \neq L$.
\end{enumerate}

Now to get the first bound on $\lambda^*$, we first let $(\lambda,u)$ be a solution pair of \eqref{eq:ode} and consider the eigenvalue problem 
 \begin{equation}\label{eq:eig_prob}
  -\left(\frac{\varphi'}{\sqrt{1+(u')^2}} \right)'  = \mu\,  \varphi,  \quad  -L<x<L; \qquad   \varphi(- L) = \varphi(L)=0,
\end{equation}
for $\varphi \in H_0^1[-L,L]$. After multiplying \eqref{eq:ode} and \eqref{eq:eig_prob} by $\varphi$ and $u$, respectively, and integrating  over $(-L,L)$, we obtain
\[
 \int_{-L}^L  \frac{\varphi' u'}{\sqrt{1+(u')^2}} \ud{x}= \int_{-L}^L\frac{\lambda \varphi}{(1-u)^2} \ud{x}, \qquad  \int_{-L}^L  \frac{\varphi' u'}{\sqrt{1+(u')^2}} \ud{x} = \int_{-L}^L\mu\,  u \, \varphi \ud{x},
\]
which yields the following solvability condition for $u$:
\[
0= \int_{-L}^L \varphi\left[\frac{\lambda }{(1-u)^2} - \mu\,  u \right]\ud{x},
\]
for all eigenpairs, $(\mu,\varphi)$, of problem \eqref{eq:eig_prob}. In particular, it holds for the first eigenvalue $\mu_1$ of problem \eqref{eq:eig_prob}, which is positive and simple, and its corresponding eigenvalue $\varphi_1$, which can---and will---be chosen to be strictly positive in $[-L,L]$. In choosing  $\varphi_1(x)>0$ for all $x$ in $[-L,L]$, the term $\lambda (1-u)^{-2}-\mu_1 u $ must be identically zero or change sign.  It is clear that it is not zero; hence, there must be a value of $u$, say $\hat{u}\in(0,1)$, where $\lambda (1-\hat{u})^{-2}= \mu_1 \hat{u}$. Since this expression must be true, we obtain $0<\lambda \leq  4\mu_1/27.$ 
However, using Rayleigh's formula, we also have
\[
 \begin{aligned}
 \mu_1 & = \inf \left\{ \int_{-L}^L \frac{|\varphi'|^2}{\sqrt{1+|u'|^2}} \  \mathrm{d}{x} : \varphi \in H_0^1[-L,L], \ \|\varphi\|_{2} = 1\right\} \\
 	& <  \inf \left\{ \int_{-L}^L |\varphi'|^2  \  \mathrm{d}{x}: \varphi \in H_0^1[-L,L], \ \|\varphi\|_{2} = 1\right\} \\
	& = \kappa_1,
 \end{aligned}
\]
where $\kappa_1 = \pi^2/(4L^2)$ is the first eigenvalue of $-\partial^2/\partial x^2$ on $[-L,L]$ with homogeneous Dirichlet boundary conditions and $\|\cdot\|_2$ is the standard $L^2$ norm on $[-L,L]$. Therefore, the value $\lambda \leq 4\mu_1/27< 4\kappa_1/27 = \pi^2/(27L^2)$.

To get the second bound  we note that if $u$ is a solution to \eqref{eq:ode}, then $(1-u)^{-2}>1$. Therefore, integrating the differential equation in \eqref{eq:ode} from $0$ to $x$ and using the aforementioned inequality gives
\[
  -\frac{u'(x)}{\sqrt{1+u'(x)^2}} > \lambda x, \qquad 0<x \leq L,
\]
where we have used that fact that $u'(0)=0$. From Lemma \ref{lem:nc1}, we know that $-u'(x) = |u'(x)|$; thus, 
\[
  \lambda x <  \frac{|u'(x)|}{\sqrt{1+|u'(x)|^2}} < 1,
\]
which upon taking $x \to L^-$ yields $ \lambda < L^{-1}$.

Hence, $\lambda^* < \min\{L^{-1},\pi^2/(27L^2)\}$.
\end{proof}

\begin{figure}[!h]
 \centering
 \subfigure[$L < L^*$]{\label{fig:case2}\includegraphics[width=.47\textwidth]{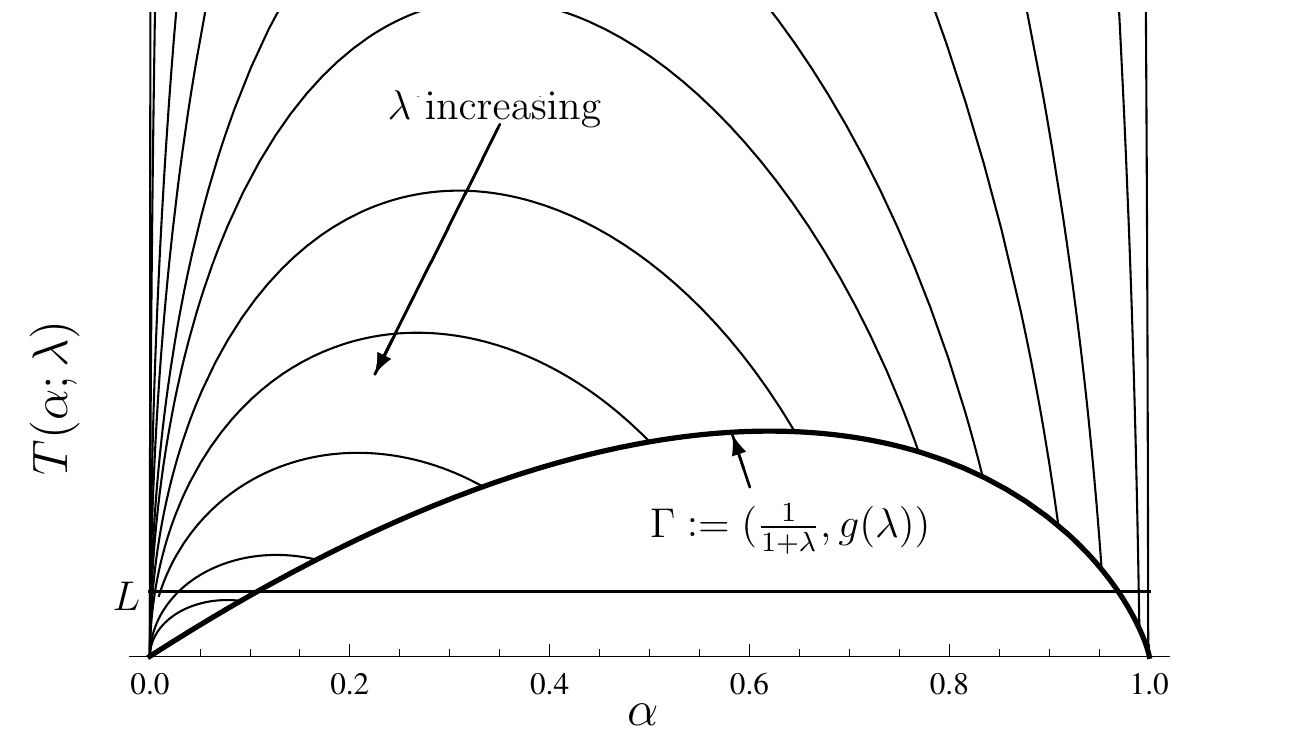}} \hspace{.04\textwidth} 
 \subfigure[$L \geq L^*$]{\label{fig:case1}\includegraphics[width=.47\textwidth]{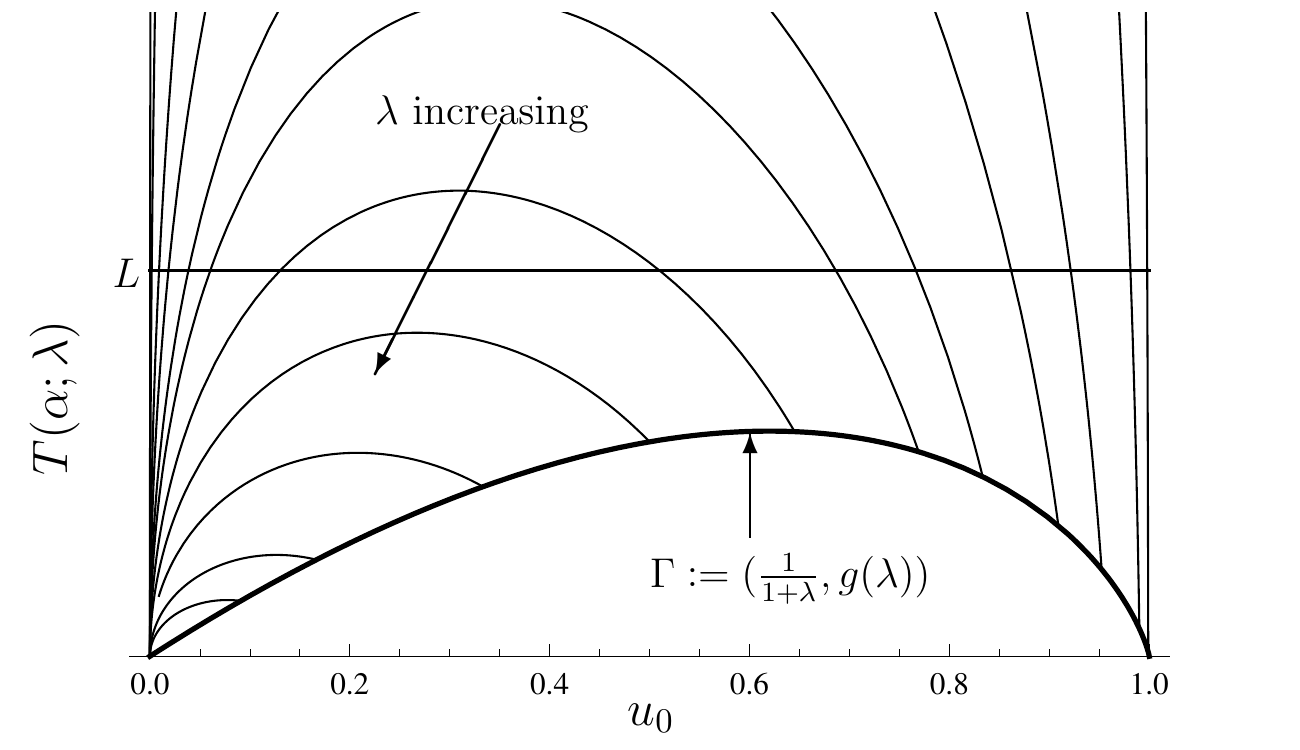}} 
 \caption{Graphical illustration of Theorem \ref{thm:maintheorem}}\label{fig:twocases}
\end{figure}

\section{Conclusion} 
In this work, we have analyzed the solution set of the one-dimensional prescribed mean curvature problem
\be\tag{\ref{eq:ode}}
\begin{aligned}
 & -\left(\frac{u'}{\sqrt{1+(u')^2}} \right)' = \frac{\lambda}{(1-u)^2}, \quad   -L<x<L; \qquad   u(-L)=u(L)= 0,
\end{aligned}
\ee
with $u <1$ in $[-L,L$], for positive $L$ and positive $\lambda$. In particular, we have shown that the solution set undergoes two bifurcations: a saddle node bifurcation and codimension 2 \emph{splitting} bifurcation, which depends on the parameters $L$ and $\lambda$. As a result of this analysis, we speculate that the latter, which is not present in the corresponding semilinear problem ($-\Delta u = \lambda^{-1}(1-u)^{-2}$ with the same boundary conditions), is due to the interplay of the mean curvature operator and finite-singularity forcing. This is because while other one-dimensional prescribed mean curvature equations have exhibited a disappearing solutions behavior (c.f. \cite{pan2009onedimensional}), none with a continuous forcing have exhibited this splitting bifurcation. With this in mind it would be desirable to fully characterize the solution set of the one-dimensional prescribed mean curvature equation with $f(u) = \lambda (1-u)^{-p}$, for $p>0$, and see if the splitting bifurcation generalizes to this class of problems.

\Appendix
\section{Bounding $H_i,$ for $i=1,2,3$}\label{sec:appendix}

By Proposition \ref{prop:nocomplexderivative} for $\alpha\in[a,1/(1+\lambda)]$, $a>0$, and $z\in(0,1)$, we deduce
\[
\begin{aligned}
   0 \leq  H_1(\alpha,z;\lambda) & = \frac{1}{2\sqrt{\lambda \alpha}}\frac{(1-\alpha z)(1-\alpha)-\lambda \alpha(1- z)}{\sqrt{1- z}\sqrt{2 (1-\alpha z)(1-\alpha) - \lambda \alpha(1- z) }} \\
	& \leq \frac{1}{2\sqrt{\lambda \alpha}}\frac{1}{\sqrt{1- z}\sqrt{2 (1-\alpha z)(1-\alpha) - \lambda \alpha(1- z) }},
\end{aligned}
\]
and
\[
\begin{aligned}
   0 \leq |H_2(\alpha,z;\lambda)| & = \sqrt{\frac{\alpha}{\lambda}}\frac{|1 +(1-2\alpha)z  +\lambda(1-z)|}{\sqrt{1- z}\sqrt{2 (1-\alpha z)(1-\alpha) - \lambda \alpha(1- z) }} \\
	& \leq  \sqrt{\frac{\alpha}{\lambda}}\frac{1 +|1-2\alpha|   +\lambda}{\sqrt{1- z}\sqrt{2 (1-\alpha z)(1-\alpha) - \lambda \alpha(1- z) }} \\
	& \leq  \sqrt{\frac{\alpha}{\lambda}}\frac{2   +\lambda}{\sqrt{1- z}\sqrt{2 (1-\alpha z)(1-\alpha) - \lambda \alpha(1- z) }}.
\end{aligned}
\]
Moreover from Proposition \ref{prop:nocomplexderivative},
\[
 2 (1-\alpha z)(1-\alpha) - \lambda \alpha(1- z) \geq \frac{\lambda}{1+\lambda}(1-\alpha z +\lambda \alpha z) \geq \frac{\lambda^2 \alpha }{1+\lambda}z >0,
\]
which implies 
\begin{equation}\label{eq:boundG1}
\begin{aligned}
   0 \leq H_1 & < \frac{\sqrt{1+\lambda}}{2 \lambda^{3/2} \alpha}\frac{1}{\sqrt{z}\sqrt{1- z}} < \frac{\sqrt{1+\lambda}}{2 \lambda^{3/2} a}\frac{1}{\sqrt{z}\sqrt{1- z}}
\end{aligned}
\ee
and
\begin{equation}\label{eq:boundG2}
\begin{aligned}
   0 \leq |H_2| & \leq  (2   +\lambda)\frac{\sqrt{1+\lambda}}{\lambda^{3/2} } \frac{1}{\sqrt{z}\sqrt{1- z}}.
\end{aligned}
\ee

Again,  from Proposition \ref{prop:nocomplexderivative} for $\alpha\in[a,1/(1+\lambda)]$, $a>0$, and $z\in(0,1)$, we have
\[
\begin{aligned}
   0 \leq |H_3(\alpha,z;\lambda)| & =\sqrt{\frac{\alpha}{\lambda}}\frac{((1-\alpha z)(1-\alpha)-\lambda \alpha(1- z))|2+\lambda(1-z) + (2-4\alpha)z|}{2\sqrt{1-z}\left(2 (1-\alpha z)(1-\alpha) - \lambda \alpha(1- z)\right)^{3/2}} \\
	& \leq \sqrt{\frac{\alpha}{\lambda}}\frac{2+\lambda + 2|1-2\alpha|}{2\sqrt{1-z}\left(2 (1-\alpha z)(1-\alpha) - \lambda \alpha(1- z)\right)^{3/2}}\\
	& \leq \sqrt{\frac{\alpha}{\lambda}}\frac{4+\lambda}{2\sqrt{1-z}\left(2 (1-\alpha z)(1-\alpha) - \lambda \alpha(1- z)\right)^{3/2}}.
\end{aligned}
\]
Also, from Proposition \ref{prop:nocomplexderivative},
\[
 2 (1-\alpha z)(1-\alpha) - \lambda \alpha(1- z) \geq \frac{\lambda}{1+\lambda}(1-\alpha z +\lambda \alpha z) \geq \frac{\lambda}{1+\lambda}\left(1- \alpha z\right) >0,
\]
which implies
\begin{equation}\label{eq:boundG3}
\begin{aligned}
   0\leq |H_3| & \leq \frac{\sqrt{\alpha}(1+\lambda)^{3/2}(4+\lambda)}{2\lambda^{2}}\frac{1}{\left(1- \alpha z\right)^{3/2}} \frac{1}{\sqrt{1-z}} \\
   	& \leq \frac{(1+\lambda)(4+\lambda)}{2\lambda^{2}}\frac{1}{\left(1-  (1+\lambda)^{-1}z\right)^{3/2}} \frac{1}{\sqrt{1-z}}.
\end{aligned}
\ee

Furthermore, for $\lambda>0$
\be\label{eq:boundAppend4}
 \int_0^1 \frac{1}{\sqrt{z}\sqrt{1- z}} \ud z = \pi, \quad   \int_0^1\frac{1}{\left(1-  (1+\lambda)^{-1}z\right)^{3/2}} \frac{1}{\sqrt{1-z}} \ud z = 2 \left(1 + \lambda^{-1}\right),
\ee
which means that $1/(\sqrt{z}\sqrt{1- z})$ and $(1- (1+\lambda)^{-1}z)^{-3/2}/\sqrt{1- z}$ are in $L^1(0,1)$.

\bibliographystyle{siam}	
\bibliography{global_ref}		

\end{document}